\documentclass[a4paper,11pt]{article}
\usepackage[T1]{fontenc}
\usepackage[utf8]{inputenc}
\usepackage{amsmath}
\usepackage{amsthm}
\usepackage{amssymb}
\usepackage{amsfonts}
\usepackage{graphics}
\usepackage{color}
\usepackage[all,cmtip]{xy}
\newcounter{theorem}

\newtheorem{theo}[theorem]{Theorem}%
\newtheorem{prop}[theorem]{Proposition}%
\newtheorem{defi}[theorem]{Definition}%
\newtheorem{lemm}[theorem]{Lemma}%
\newtheorem{coro}[theorem]{Corollary}%
\newtheorem{rema}[theorem]{Remark}%
\newtheorem{exam}[theorem]{Example}%

\newcommand{\myar}{\ar@[|(3)]}
\newcommand{\dar}{\ar@{-->}}

\def\odc{[\![}
\def\fdc{]\!]}

\def\ac{\mathcal A}
\def\bc{\mathcal B}
\def\cc{\mathcal C}
\def\dc{\mathcal D}
\def\ec{\mathcal E}
\def\fc{\mathcal F}
\def\gc{\mathcal G}

\def\mc{\mathcal M}
\def\pc{\mathcal P}
\def\tc{\mathcal T}
\def\uc{\mathcal U}
\def\vc{\mathcal V}
\def\xc{\mathcal X}
\def\yc{\mathcal Y}
\def\zc{\mathcal Z}

\def\eef{{\mathbb E}}
\def\nnf{{\mathbf N}}
\def\ppf{{\mathbb P}}

\def\rrf{{\mathbb R}}
\def\ttf{{\mathbb T}}

\def\zzf{{\mathbf Z}}
\def\one{{\mathbf 1}}
\def\eef{{\mathbb E}}

\def\d{{\rm d}}

\def\ep{{\varepsilon}}
\def\var{{\rm Var}}
\def\sign{{\rm sign}}

\usepackage{subfigure}
\usepackage{relsize,exscale}
\font\quinzerm=cmr15
\font\dixneufrm=cmr19
\def\moy{\mathop{\kern0pt \raise-4pt \hbox{\dixneufrm M}}\limits}
\def\petmoy{\mathop{\kern0pt \raise-2pt \hbox{\quinzerm M}}\limits}
\begin{document}
\title{Complementability and maximality in different contexts: 
ergodic theory, Brownian and poly-adic filtrations}
\author{Christophe Leuridan}
\maketitle

\begin{abstract}
The notions of complementability and maximality were introduced 
in 1974 by Ornstein and Weiss in the context of the automorphisms 
of a probability space, in 2008 by Brossard and Leuridan in the 
context of the Brownian filtrations, and in 2017 by Leuridan in the 
context of the poly-adic filtrations indexed by the non-positive 
integers. We present here some striking analogies and also some differences 
existing between these three contexts.
\end{abstract}

\begin{flushleft}
\end{flushleft}

\section{Introduction}

\subsection{General context}

In the present paper, we will work with three types of objects: 
automorphisms of Lebesgue spaces, Brownian filtrations and 
filtrations indexed by $\zzf$ or $\zzf_-$; the reason for choosing 
$\zzf$ or $\zzf_-$ and to rule out $\zzf_+$ is that for discrete-time 
filtrations, the interesting phenomena occur near time $-\infty$. 

Among the invertible measure-preserving maps, Bernoulli shifts form 
a remarkable class. Similarly, the {\it product-type} filtrations (i.e., 
generated modulo the null sets by sequences of independent random variables) 
are considered as a well-understood class. The Brownian filtrations 
(generated modulo the null sets by Brownian motions) form a natural 
and widely studied class of continuous-time filtrations, although 
less simple. 

Measure-preserving maps considered here will be taken on diffuse 
Lebesgue spaces. Various equivalent definitions of Lebesgue spaces 
are available. A simple definition of a Lebesgue space is a probability 
space which is isomorphic modulo the null sets to the union of some 
sub-interval of $[0,1]$, endowed with the Lebesgue $\sigma$-field 
and the Lebesgue measure, and a countable set of atoms. 
Most of the time, the Lebsegue space considered is non-atomic, so 
the sub-interval is $[0,1]$ itself. The class of Lebesgue spaces 
includes the completion of every Polish space.
See~\cite{Kalikow - McCutcheon} to find the main properties of 
Lebesgue spaces or~\cite{de la Rue} to get equivalent definitions.  
Working on Lebesgue spaces 
provides non-trivial measurability results, existence of generators...  
We recall in section~\ref{Annex} the definitions and the main properties 
of partitions, generators, entropy used in the present paper. 

Similarly, the filtrations considered here will be defined on a 
standard Borel probability space $(\Omega,\fc,\ppf)$, i.e. $(\Omega,\fc)$ 
is the Borel space associated to some Polish space, to ensure the existence 
of regular conditional probabilities.
Given two sub-$\sigma$-fields $\ac$ and $\bc$, the inclusion 
$\ac \subset \bc \mod \ppf$ means that for every $A \in \ac$, 
there exists $B \in \bc$ such that $\ppf(A \triangle B) = 0$. 
We say that $\ac$ and $\bc$ are equal modulo the null sets (or modulo $\ppf$) 
when $\ac \subset \bc \mod \ppf$ and $\bc \subset \ac \mod \ppf$. 
We do not systematically complete the $\sigma$-fields to avoid troubles 
when working with conditional probabilities.  

\subsection{Reminders on filtrations indexed by $\zzf$ or $\zzf_-$}

We now recall some classical but less known definitions and facts 
on filtrations. Given a filtration $(\fc_n)_n$ indexed by $\zzf$ or 
$\zzf_-$, one says that $(\fc_n)_n$ {\it is product-type} if $(\fc_n)_n$
can be generated modulo $\ppf$ by some sequence $(I_n)_n$ of 
(independent) random variables. 

One says that $(\fc_n)_n$ {\it has independent increments} 
if there exists a sequence $(I_n)_n$ of random variables 
such that for every $n$ in $\zzf$ or $\zzf_-$, 
$$\fc_n = \fc_{n-1} \vee \sigma(I_n) \mod \ppf \text{ and } 
I_n \text{ is independent of } \fc_{n-1}.$$ 
Such a sequence $(I_n)_n$ is called a {\it sequence of innovations} 
and is necessarily a sequence of independent random variables.  

One says that $(\fc_n)_n$ is {\it $(a_n)_n$-adic} when it admits some sequence 
$(I_n)_n$ of innovations such that each $I_n$ is uniformly distributed 
on some finite set with size $a_n$. One says that $(\fc_n)_n$ 
is {\it poly-adic} when $(\fc_n)_n$ is $(a_n)_n$-adic for some sequence 
$(a_n)_n$ of positive integers, called {\it adicity}. 
  
One says that $(\fc_n)_n$ is {\it Kolmogorovian} if the tail $\sigma$-field
$\fc_{-\infty} := \bigcap_n \fc_n$
is trivial (i.e., contains only events with probability $0$ or $1$). 

By the definition and by Kolmogorov's zero-one law, any filtration 
indexed by $\zzf$ or $\zzf_-$ must have independent increments and must be 
Kolmogorovian to be product-type. But Vershik showed in~\cite{Vershik}
that the converse is not true. A simple counter-example is given by 
Vershik's decimation process (example 2 in \cite{Vershik}). 
Actually Vershik worked with decreasing 
sequences of measurable partitions indexed by $\zzf_+$
and this frame was translated into filtrations indexed by $\zzf_-$
by M.~\'Emery and W.~Schachermayer~\cite{Emery - Schachermayer}. 

\subsection{$K$-automorphisms}

The Kolmogorov property for filtrations indexed by $\zzf$ or $\zzf_-$
has an analogue for dynamical systems, although the definition is 
less simple in this frame: one says that an automorphism $T$ 
of a probability space $(Z,\zc,\pi)$ is a {\it $K$-automorphism} 
(or that $T$ has completely positive entropy) if 
for every $A \in \zc$, one has $h(T,\{A,A^c\})>0$ whenever $0<\pi(A)<1$.
This condition is nothing but the triviality of the $\sigma$-field
$$\Pi(T) := \{A \in \zc : h(T,\{A,A^c\})=0\},$$
called {\it Pinsker's factor}. Actually, the `events' of Pinsker's 
factor can be seen as the 
`asymptotic events'. Indeed, if $\gamma$ is a countable generator 
of $(Z,\zc,\pi,T)$, then 
$$\Pi(T) = \overline{\bigcap_{n \ge 0} \bigvee_{k \ge n} T^{-k}\gamma}
= \overline{\bigcap_{n \ge 0} \bigvee_{k \ge n} T^k\gamma},$$
where the upper bar indicates completion with regard to $\pi$. 
To make the analogy clearer, set $\gamma=\{A_\lambda , \lambda \in \Lambda\}$. 
For each $x \in Z$, call $f(x) \in \Lambda$ the only index $\lambda$ such that
$x \in A_\lambda$. 
For every $k \in \zzf$, the $\sigma$-field generated by $T^{-k}\gamma$ 
is the $\sigma$-field associated to $f \circ T^k$ viewed as a $\Lambda$-valued
random variable on $(Z,\zc,\pi)$. Therefore, $\Pi(T)$ is the asymptotic 
$\sigma$-field generated by the sequence $(f \circ T^k)_{k \ge 0}$.

\subsection{Content of the paper}

We have just viewed the analogy between the `Kolmogorovianity' 
of a filtration indexed by $\zzf_-$ and the $K$-property of an 
automorphism of a Lebesgue space.  
 
The next section is devoted to a parallel presentation of analogous
notions and results in the three following contexts: 
automorphisms of Lebesgue spaces, filtrations indexed by $\zzf_-$ 
and Brownian filtrations. We investigate two notions 
- complementability and maximality - involving factors 
or poly-adic immersed filtrations or Brownian immersed filtrations 
according to the context. 
The results presented are essentially due to 
Ornstein and Weiss~\cite{Ornstein - Weiss}, 
Ornstein~\cite{Ornstein}, and Thouvenot~\cite{Thouvenot} 
for automorphisms of Lebesgue spaces; 
They come from~\cite{Leuridan} for filtrations indexed by $\zzf_-$.  
They are due to Brossard, \'Emery and Leuridan~\cite{Brossard - Leuridan,Brossard - Emery - Leuridan,Brossard - Emery - Leuridan 2}
for Brownian filtrations. 

Section~\ref{conditions for maximality: proofs} provides proofs of 
results on maximality which are not easy yo find in the literature. 
With some restrictions on the nature of the complement, 
complementability implies maximality. 
Section~\ref{complementability implies maximality: proofs} is devoted 
to the proof of this implication. 
The converse was already known to be false for factors of automorphisms 
of Lebesgue spaces and for poly-adic immersed filtrations. 
In section~\ref{maximal but non-complementable Brownian filtration}, 
we provide a counter-example in the context of 
Brownian filtrations. The construction relies on a counter-example 
for poly-adic immersed filtrations which is inspired by non-published 
notes of Tsirelson~\cite{Tsirelson}. 

In spite of the similitude of the notions regardless the context, 
some differences exist. 
In section~\ref{complementable filtration yielding a complementable factor}, 
we provide a non-complementable filtration (associated to a stationary process)
yielding a complementable factor. This example is inspired by Vershik's 
decimation process (example 2 in \cite{Vershik}). 

In section~\ref{Annex}, we recall the definitions and the main properties 
of partitions, generators, entropy used in the present paper.

\section{Parallel notions and results}

\subsection{Factors and Immersed filtrations}\label{factors}

Given an invertible measure preserving map $T$ of a Lebesgue space 
$(Z,\zc,\pi)$, we call {\it factor of $T$}, or more rigorously a factor 
of the dynamical system $(Z,\zc,\pi,T)$, any sub-$\sigma$-field 
$\bc$ of $\zc$ such that $T^{-1}\bc = \bc = T\bc \mod \pi$. 
Actually, the factor is the dynamical system $(Z,\bc,\pi|_\bc,T)$, 
which will be abbreviated in $(T,\bc)$ in the present paper. 
This definition of a factor is equivalent to the usual one.
\footnote{Actually, Rokhlin's theory ensures that if $\bc$ is a factor of 
a Lebesgue space $(Z,\zc,\pi,T)$, then there exists a map $f$ from $Z$ to 
some Polish space $E$ such that $\bc$ is generated up to the negligible events
by the map $\Phi : x \mapsto (f(T^k(x)))_{k \in \zzf}$ from $Z$ to the product 
space $E^\zzf$. Call $\nu = \Phi(\pi) = \pi \circ \Phi^{-1}$ the image measure 
of $\mu$ by $\Phi$ . Then the completion $(E^\zzf,\bc(E^\zzf),\nu)$ is a 
Lebesgue space, the shift operator 
$S : (y_k)_{k \in \zzf} \mapsto (y_{k+1})_{k \in \zzf}$ is an automorphism of 
$E^\zzf$, and $S \circ \Phi = \Phi \circ T$. 

Conversely, if $(Y,\yc,\nu,S)$ is a dynamical system and $\Phi$ a 
measurable map from $Z$ to $Y$ such that $\Phi(\pi)=\nu$ and 
$S \circ \Phi = \Phi \circ T$, then the $\sigma$-field 
$\Phi^{-1}(\yc)$ is a factor of $(Z,\zc,\pi,T)$.}

\vfill\eject

Given two filtrations $(\uc_t)_{t \in \ttf}$ and $(\zc_t)_{t \in \ttf}$ on some 
probability space $(\Omega,\ac,\ppf)$, indexed by a common subset 
$\ttf$ of $\rrf$, one says that $(\uc_t)_{t \in \ttf}$ is {\it immersed in} 
$(\zc_t)_{t \in \ttf}$ if every martingale in $(\uc_t)_{t \in \ttf}$ is still a  
martingale in $(\zc_t)_{t \in \ttf}$. The notion of immersion is stronger than 
the inclusion. Actually, $(\uc_t)_{t \in \ttf}$ is immersed 
in $(\zc_t)_{t \in \ttf}$ if and only if the two conditions below hold:
\begin{enumerate}
\item for every $t \in \ttf$, $\uc_t \subset \zc_t$. 
\item for every $s<t$ in $\ttf$, $\uc_t$ and $\zc_s$ are independent 
conditionally on $\uc_s$.  
\end{enumerate}
The additional condition means that the largest filtration does not give 
information in advance on the smallest one. We also make the useful 
following observation. 

\begin{lemm}\label{immersion and final sigma-field}
Assume that $(\uc_t)_{t \in \ttf}$ is immersed in $(\zc_t)_{t \in \ttf}$. 
Then $(\uc_t)_{t \in \ttf}$ is completely determined (up to null sets) 
by its final $\sigma$-field 
$$\uc_\infty := \bigvee_{t \in \ttf} \uc_t.$$
More precisely, $\uc_t = \uc_\infty \cap \zc_t \mod \ppf$ for every $t \in \ttf$.
In particular, if $\uc_\infty = \zc_\infty \mod \ppf$, then 
$\uc_t = \zc_t \mod \ppf$ for every $t \in \ttf$.
\footnote{The inclusion $\uc_t = \uc_\infty \cap \zc_t$ is immediate. 
To prove the converse, take $A \in \uc_\infty \cap \zc_t$. Since $\uc_\infty$ 
and $\zc_t$ are independent conditionally on $\uc_t$, we get 
$\ppf[A|\uc_t] = \ppf[A|\zc_t] = \one_A$ a.s., so $A \in \uc_t \mod \ppf$.}
\end{lemm}


When one works with Brownian filtrations, i.e. with filtrations 
generated by Brownian motions, then the immersion has many equivalent 
translations. The next statements are very classical (close statements 
are proved in~\cite{Attal - Burdzy - Emery - Hu}) and they rely on the 
stochastic calculus and the predictable representation property of 
Brownian filtrations. 

\begin{prop}\label{immersion of a Brownian filtration}
Let $(B_t)_{t \ge 0}$ be a finite-dimensional 
Brownian motion adapted to some filtration 
$(\zc_t)_{t \ge 0}$, and $(\bc_t)_{t \ge 0}$ its natural filtration. 
The following statements are equivalent.  
\begin{enumerate}
\item $(B_t)_{t \ge 0}$ is a martingale in $(\zc_t)_{t \ge 0}$. 
\item $(\bc_t)_{t \ge 0}$ is immersed in $(\zc_t)_{t \ge 0}$. 
\item For every $t \ge 0$, the process $B_{t+\cdot}-B_t$ is 
independent of $\zc_t$.  
\item $(B_t)_{t \ge 0}$ is a Markov process in $(\zc_t)_{t \ge 0}$.
\end{enumerate}
\end{prop}

\begin{defi}\label{Brownian motion in a filtration}
When these statements hold, we say that $(B_t)_{t \ge 0}$ is a 
Brownian motion in the filtration $(\zc_t)_{t \ge 0}$. 
\end{defi}

Note the analogy between the following two results. 

\begin{theo}[Ornstein \cite{Ornstein-1971}]
Every factor of a Bernoulli shift is equivalent to a Bernoulli shift.  
\end{theo}

\begin{theo}[Vershik \cite{Vershik}]\label{Vershik}
If $(\zc_n)_{n \le 0}$ is a product-type filtration such that 
the final sigma-field $\zc_0$ is essentially separable, then 
every poly-adic filtration immersed in $(\zc_n)_{n \le 0}$ is product-type.  
\end{theo}

One-dimensional Brownian filtrations can be viewed as continuous time 
versions of dyadic product-type filtrations. In this analogy, the predictable 
representation property of the continuous-time filtration corresponds 
to the dyadicity of the discrete-time filtration. 
Yet, the situation is much more involved when one works with Brownian 
filtrations, and the following question remains open.

If a filtration $(\fc_t)_{t \ge 0}$ is immersed in some (possibly 
infinite-dimensional) Brownian filtration and has the predictable 
representation property with regard to some one-dimensional Brownian 
motion $\beta$ (i.e., each martingale in $(\fc_t)_{t \ge 0}$ can be 
obtained as a stochastic integral with regard to $\beta$), then is 
$(\fc_t)_{t \ge 0}$ necessarily a Brownian filtration?

A partial answer was given by \'Emery (it follows from 
corollary 1 in~\cite{Emery}). 

\begin{theo}[\'Emery~\cite{Emery}]
\label{Brownian after $0$ and immersed in a Brownian filtration}
Let $d \in \nnf \cup\{+\infty\}$. 
Assume that the filtration $(\fc_t)_{t \ge 0}$ is $d$-Brownian after $0$, 
i.e., there exists a $d$-dimensional Brownian motion 
$(B_t)_{t \ge 0}$ in $(\fc_t)_{t \ge 0}$ such that for every 
$t \ge \ep > 0$, $\fc_t$ is generated by $\fc_\ep$ and the 
increments $(B_s-B_\ep)_{\ep \le s \le t}$. If $(\fc_t)_{t \ge 0}$
is immersed in some (possibly infinite-dimensional) Brownian filtration, 
then $(\fc_t)_{t \ge 0}$ is a $d$-dimensional Brownian filtration.  
\end{theo}
 
In the statement, the role of the stronger hypothesis that $(\fc_t)_{t \ge 0}$ 
is Brownian after $0$ is to guarantee that the difficulties arise only 
at time $0+$, so the situation gets closer to filtrations indexed by 
$\zzf$ or $\zzf_-$, for which the difficulties arise only at time $-\infty$.  

\subsection{Complementability}

By complementability, we will mean the existence of some independent 
complement, although we will have to specify the nature of the complement.  

The following definition is abridged from~\cite{Ornstein - Weiss}.

\begin{defi}
Let $(Z,\zc,\pi,T)$ be a Lebesgue dynamical system and $\bc$ be a 
factor of $T$. 
One says that $\bc$ is complementable if $\bc$ possesses an independent 
complement in $(Z,\zc,\pi,T)$, i.e. a factor $\cc$ of $T$ which is 
independent of $\bc$ (with regard to $\pi$) such that 
$\bc \vee \cc = \zc \mod \pi$. 
\end{defi}

If $(Z,\zc,\pi,T)$ is the direct product of two dynamical systems 
$(Z_1,\zc_1,\pi_1,T_1)$ and $(Z_2,\zc_2,\pi_2,T_2)$, then 
$\zc_1 \otimes \{\emptyset,Z_2\}$ and $\{\emptyset,Z_1\} \otimes \zc_2$ 
are factors of $(Z,\zc,\pi,T)$ and each of them is a complement of 
the other one. Now, let us look at a counterexample.

\begin{exam}\label{head and tail}
Let $T$ be the Bernoulli shift
on $Z = \{-1,1\}^\zzf$ endowed with product $\sigma$-field $\zc$ 
and the uniform law. The map $\Phi : Z \to Z$ defined by 
$\Phi((x_n)_{n \in \zzf}) = (x_{n-1}x_n)_{n \in \zzf}$ commutes with $T$, 
so $\Phi^{-1}(\zc)$ is a factor of $T$. Call $p_0 : Z \to \{-1,1\}$ 
the canonical projection defined by $p_0((x_n)_{n \in \zzf}) = x_0$.  
Then the $\sigma$-field $p_0^{-1}(\zc)$ is an independent complement of 
$\Phi^{-1}(\zc)$, but this complement is not a factor. Actually, we will 
come back to this example after definition~\ref{maximal Brownian filtration}
to show as an application of theorem~\ref{implication: factor case}
that no factor can be an independent complement of $\Phi^{-1}(\zc)$.  
\end{exam}

We now define the notion of complementability in the world of filtrations.   

\begin{defi}
Consider two filtrations $(\uc_t)_{t \in \ttf}$ and $(\zc_t)_{t \in \ttf}$ on 
some probability space $(\Omega,\ac,\ppf)$, indexed by a common subset 
$\ttf$ of $\rrf$. One says that $(\uc_t)_{t \in \ttf}$ is complementable 
in $(\zc_t)_{t \in \ttf}$ if there exists a filtration $(\vc_t)_{t \in \ttf}$ 
such that for every $t \in \ttf$, $\uc_t$ and $\vc_t$ are independent 
and $\uc_t \vee \vc_t = \zc_t \mod \ppf$. 
\end{defi}

Since independent enlargements of a filtration always produce 
filtrations in which the initial filtration is immersed, 
$(\uc_t)_{t \in \ttf}$ needs to be immersed in $(\zc_t)_{t \in \ttf}$ 
to possess an independent complement. 

We will use many times the next result, abridged from~\cite{Leuridan}. 

\begin{prop}\label{complementability and conditionning} 
Keep the notations of the last definition. Let $U$ be a random 
variable valued in some measurable space $(E,\ec)$, 
such that $\sigma(U) = \bigvee_{t \in \ttf}\uc_t$, and $(\ppf_u)_{u \in E}$ 
a regular version of the conditional probability $\ppf$ given $U$. 
Assume that $(\uc_t)_{t \in \ttf}$ is complementable in $(\zc_t)_{t \in \ttf}$ by 
a filtration $(\vc_t)_{t \in \ttf}$. 
Then for $U(\ppf)$-almost every $u \in E$, the filtered probability space 
$(\Omega,\ac,\ppf_u,(\zc_t)_{t \le \ttf})$ is isomorphic to the filtered probability 
space $(\Omega,\ac,\ppf,(\vc_t)_{t \le \ttf})$. 
\end{prop}

Let us give applications of the last result, that will be used in the 
present paper. 

\begin{coro}\label{conditionning} (Particular cases)
\begin{itemize}
\item If a filtration $(\uc_n)_{n \le 0}$ is complementable in 
$(\zc_n)_{n \le 0}$ by some product-type filtration, then for 
$U(\ppf)$-almost every $u \in E$, 
$(\zc_n)_{n \le 0}$ is product-type under $\ppf_u$. 
\item If a filtration $(\uc_t)_{t \ge 0}$ is complementable in 
$(\zc_t)_{t \ge 0}$ by some Brownian filtration, then for 
$U(\ppf)$-almost every $u \in E$, 
$(\zc_t)_{t \ge 0}$ is a Brownian filtration under $\ppf_u$.  
\end{itemize}
\end{coro}

Determining whether a $1$-dimensional Brownian filtration immersed in 
a $2$-dimen\-sional Brownian filtration is complementable or not is often 
difficult. Except trivial cases, the only known cases are related to 
skew-product decomposition of the planar Brownian motion, 
see~\cite{Brossard - Emery - Leuridan 2}. 

\subsection{Maximality}

The definition of the maximality requires a tool to measure 
the quantity of information. When one works with factors of 
an automorphism of a Lebesgue space, the quantity of information 
is the entropy. When one works with poly-adic filtrations, the 
quantity of information is the sequence of positive integers 
giving the adicity. When one works with Brownian filtrations, 
the quantity of information is the dimension of any generating 
Brownian motion. The classical statements below show hove these
quantities vary when one considers a factor, a poly-adic 
immersed filtration, or a Brownian immersed filtration.  

\begin{rema}\label{quantity of information} 
(Quantity of information in subsystems)
\begin{enumerate}
\item If $\bc$ is a factor of $(Z,\zc,\pi,T)$, then $h(T,\bc) \le h(T)$. 
\item If a $(b_n)_{n \le 0}$-adic filtration is immersed in 
an $(r_n)_{n \le 0}$-adic filtration, then $b_n$ divides $r_n$ for every $n$.
\item If a $m$-dimensional Brownian filtration is immersed in a 
$n$-dimensional Brownian filtration, then $m \le n$. 
\end{enumerate}
\end{rema}  

The first statement is very classical. The second one is proved 
in~\cite{Leuridan}. The last one is classical and shows that the 
dimension of a Brownian filtration makes sense; a proof is given 
in the footnote. 
\footnote{Let $Z$ be a $n$-dimensional Brownian motion and $B$ be a 
$m$-dimensional Brownian motion in $\fc^Z$. Then one can find 
an $\fc^Z$-predictable process $M$ taking values in the set 
of all $p \times n$ real matrices whose lines form an 
orthonormal family, such that $B = \int_0^\cdot M_s \d Z_s.$
In particular, the $m$ lines of each matrix $M_s$ are independent and 
lie in a $n$-dimensional vector space, so $m \le n$.}  

Let us give precise definitions, respectively abridged from~\cite{Ornstein},
~\cite{Leuridan} and~\cite{Brossard - Leuridan} 
or~\cite{Brossard - Emery - Leuridan}. 

\begin{defi}\label{maximal factor}
Let $(Z,\zc,\pi,T)$ be a Lebesgue dynamical system and $\bc$ be a 
factor of $T$. 
One says that $\bc$ is maximal if $(T,\bc)$ has a finite entropy 
and if for any factor $\ac$, the conditions $\ac \supset \bc$ and 
$h(T,\ac) = h(T,\bc)$ entail $\ac = \bc$ modulo null sets.  
\end{defi}

\begin{defi}\label{maximal poly-adic filtration}
Let $(\bc_n)_{n \le 0}$ be a $(b_n)_{n \le 0}$-adic filtration immersed in some 
filtration $(\zc_n)_{n \le 0}$. One says that $(\bc_n)_{n \le 0}$ is maximal in
$(\zc_n)_{n \le 0}$ if every $(b_n)_{n \le 0}$-adic filtration immersed in
$(\zc_n)_{n \le 0}$ and containing $(\bc_n)_{n \le 0}$ 
is equal to $(\bc_n)_{n \le 0}$ modulo null events.  
\end{defi}

\begin{defi}\label{maximal Brownian filtration}
Let $(\bc_t)_{t \ge 0}$ be a $d$-dimensional Brownian filtration immersed in 
some filtration $(\zc_t)_{t \ge 0}$. One says that $(\bc_t)_{t \ge 0}$ is maximal 
in $(\zc_t)_{t \ge 0}$ if every $d$-dimensional Brownian filtration immersed 
in $(\zc_t)_{t \ge 0}$ and containing $(\bc_t)_{t \le 0}$ is equal to 
$(\bc_t)_{t \ge 0}$ modulo null events.  
\end{defi}

Let us come back to example~\ref{head and tail} in which $T$ be the 
Bernoulli shift on $Z = \{-1,1\}^\zzf$ endowed with product 
$\sigma$-field $\zc$ and the uniform law. Since the map 
$\Phi : Z \to Z$ defined by $\Phi((x_n)_{n \in \zzf}) = (x_{n-1}x_n)_{n \in \zzf}$ 
commutes with $T$ and preserves the uniform law on $Z$, 
the factor $(T,\Phi^{-1}(\zc))$ is a Bernoulli $(1/2,1/2)$ shift 
like $T$ itself. The factor $\Phi^{-1}(\zc)$ is strictly contained 
in $\zc$ but has the same (finite) entropy as $T$, so it is not maximal.
But every factor of $T$ is a $K$-automorphism since $T$ is. Hence, 
theorem~\ref{implication: factor case} will show that the factor 
$\Phi^{-1}(\zc)$ is not complementable. 
  
This example above can be abridged in the context of filtrations 
indexed by the relative integers : consider a sequence 
$(\xi_n)_{n \in \zzf}$ of independent uniform random variables 
taking values in $\{-1,1\}$. Then the sequence 
$(\eta_n)_{n \in \zzf} := \Phi((\xi_n)_{n \in \zzf})$ 
has the same law as $(\xi_n)_{n \in \zzf}$. One checks that the inclusions 
$\fc^\eta_n \subset \fc^\xi_n$ are strict modulo $\ppf$, although the 
tail $\sigma$-field $\fc^\xi_{-\infty}$ are trivial and although
$(\eta_n)_{n \in \zzf}$ is an innovation sequence for $(\fc^\xi_n)_{n \in \zzf}$.
Actually, one bit of information is lost 
when one transforms $(\xi_n)_{n \in \zzf}$ into $(\eta_n)_{n \in \zzf}$: 
for each $n_0 \in \zzf$, the value $\xi_{n_0}$ is independent of 
$(\eta_n)_{n \in \zzf}$, and the knowledge of $\xi_{n_0}$ and 
$(\eta_n)_{n \in \zzf}$ is sufficient to recover $(\xi_n)_{n \in \zzf}$.
The paradox is that this loss of information is asymptotic 
at time $-\infty$ but invisible when one looks at 
$\fc^\xi_{-\infty}$ and $\fc^\eta_{-\infty}$.

The situation can be much more complex when one works with Brownian 
filtrations. For example, consider a linear Brownian motion $W$.  
Since $W$ spends a null-time at $0$, the stochastic integral
$$W' = \int_0^\cdot \mathrm{sgn}(W_s) \d W_s = |W|-L$$
(where $L$ denotes the local time of $W$ at $0$) 
is still a linear Brownian motion. The natural filtration $\fc^{W'}$ 
is immersed and strictly included in $\fc^{W}$, therefore it is not 
maximal in $\fc^{W}$. Actually, $W'$ generates the same filtration 
as $|W|$ up to null events, so the L\'evy transformation - which 
transforms the sample paths of $W$ into the sample paths of $W'$
 - forgets the signs of all excursions of $W$, which are independent 
of $|W|$. Here, the loss of information occurs at each beginning of 
excursion of $W$, and not at time $0+$. 

\subsection{Necessary or sufficient conditions for maximality}
\label{conditions for maximality}

Given a finite-entropy factor, a poly-adic immersed filtration or a 
Brownian immersed filtration, one wishes to enlarge it to get a maximal one 
having the same entropy, adicity or dimension. This leads to the 
following constructions, abridged 
from~\cite{Ornstein,Leuridan,Brossard - Emery - Leuridan}. 
In the next three propositions, the bars above the $\sigma$-fields 
indicate completions with regard to $\pi$ or $\ppf$.   

\begin{defi}
Let $(Z,\zc,\pi)$ be a probability space, $T$ be an invertible 
measure-preserving map on $(Z,\zc,\pi)$ and $\bc$ be a factor 
with finite entropy. The conditional Pinsker factor associated 
to $\bc$ is defined by $$\bc' := \{A \in \zc : h(T,\{A,A^c\}|\bc)=0\}.$$
where 
$$h(T,\{A,A^c\}|\bc) 
= \lim_{n \to +\infty} \frac{1}{n} 
H \Big( \bigvee_{k=0}^{n-1} \{T^{-k}A,T^{-k}A^c\} \Big| \bc \Big).$$
\end{defi}

\begin{prop}\label{conditional Pinsker factor}
The collection $\bc'$ thus defined is the largest factor containing 
$\bc$ and having the same entropy as $\bc$. In particular, $\bc'$ is maximal. 
\end{prop}

\begin{prop}\label{slightly larger factor}
Furthermore, assume $(Z,\zc,\pi,T)$ is a Lebesgue dynamical space, that $T$ is 
aperiodic~\footnote{Aperiodicity of $T$ means that 
$\pi\{z \in Z : \exists n \ge 1, T^n(z)=z\}=0$. 
We make this assumption to ensure the existence of generator.} 
and has finite entropy. Then for every generator $\gamma$ of $T$,
$$\bc'= 
\overline{\bigcap_{n \ge 0} \Big(\bc \vee \bigvee_{k \ge n}T^{-k}\gamma \Big)}.$$
\end{prop}

\begin{prop}\label{slightly larger poly-adic filtration}
Let $(\bc_n)_{n \le 0}$ be a $(b_n)_{n \le 0}$-adic filtration immersed 
in some filtration $(\zc_n)_{n \le 0}$. Then $(\bc_n)_{n \le 0}$ is immersed in the 
filtration $(\bc'_n)_{n \le 0}$ defined by
$$\bc'_n := \overline{\bigcap_{s \le 0} (\bc_n \vee \zc_s)}.$$
Moreover $(\bc'_n)_{n \le 0}$ is the largest $(b_n)_{n \le 0}$-adic filtration 
containing $(\bc_n)_{n \le 0}$ and immersed in $(\zc_n)_{n \le 0}$.  
In particular, $(\bc'_n)_{n \le 0}$ is maximal in $(\zc_n)_{n \le 0}$. 
\end{prop}

\begin{prop}\label{slightly larger Brownian filtration}
Let $(\bc_t)_{t \ge 0}$ be a $d$-dimensional Brownian filtration immersed 
in some Brownian filtration $(\zc_t)_{t \ge 0}$. Then $(\bc_t)_{t \ge 0}$ is 
immersed in the filtration $(\bc'_t)_{t \ge 0}$ defined by
$$\bc'_t := \overline{\bigcap_{s > 0} (\bc_t \vee \zc_s)}.$$
Moreover $(\bc'_t)_{t \ge 0}$ is a $d$-dimensional Brownian filtration 
immersed in $(\zc_t)_{t \ge 0}$.
\end{prop}

Be careful: when $\fc$ is a sub-$\sigma$-field and $(\gc_n)_{n \ge 0}$ is a 
non-increasing sequence of sub-$\sigma$-fields of a probability space 
$(\Omega,\tc,\ppf)$, the trivial inclusion
$$\fc \vee \Big( \bigcap_{n \ge 0} \gc_n \Big) \subset 
\bigcap_{n \ge 0} \big(\fc \vee \gc_n \big)$$
may be strict modulo $\ppf$. Equality modulo $\ppf$ holds when $\fc$ 
and $\gc_0$ are independent (see corollary~\ref{independence allows exchange}). 
Von Weizs\"acker provides involved characterizations in \cite{Weizsacker}.
Therefore, the $\sigma$-fields $\bc'$, $\bc'_n$ and $\bc'_t$ considered in 
propositions~\ref{slightly larger factor},
\ref{slightly larger poly-adic filtration},
and~\ref{slightly larger Brownian filtration}
can be strictly larger than the $\sigma$-fields 
$\overline{\bc \vee \Pi(T)}$, $\overline{\bc_n \vee \zc_{-\infty}}$ 
and $\overline{\bc_t \vee \zc_{0+}} = \overline{\bc_t}$ respectively. 

Note the analogy between the formulas in 
propositions~\ref{slightly larger factor},
\ref{slightly larger poly-adic filtration},
and~\ref{slightly larger Brownian filtration}.
In these three contexts, we must have $\bc' = \bc$ 
up to null sets for $\bc$ to be maximal. 
Moreover, applying the same procedure to $\bc'$ leads to $\bc'' = \bc'$. 
Hence the condition $\bc' = \bc$ up to null sets is also sufficient 
for $\bc$ to be maximal in the first two cases (factors of finite-entropy 
aperiodic Lebesgue automorphisms and poly-adic filtrations). 
But once again, the situation is more complex when one works 
with Brownian filtrations, since the filtration $\bc'$ may be 
non-maximal. Here is a counter-example (the proof will be given in 
section~\ref{conditions for maximality: proofs}).

\begin{exam}\label{Levy}
Let $X$ be a linear Brownian motion in some filtration $\zc$. Set
$$B = \int_0^\cdot \mathrm{sgn}(X_s) \d X_s,$$
and call $\xc$ and $\bc$ the natural filtrations of $X$ and $B$. 
If $\xc$ is maximal in $\zc$, then the filtration 
$\bc'$ defined by proposition~\ref{slightly larger Brownian filtration}
coincides with $\bc$ up to null events. Therefore, the filtration $\bc'$ 
(included in $\xc$) cannot be maximal in $\zc$. 
\end{exam}

Actually, the maximality of Brownian filtrations is not 
an asymptotic property at $0+$, unlike the almost sure equality 
$\bc' = \bc$. To try to produce a maximal Brownian filtration 
containing a given Brownian filtration, one should perform the 
infinitesimal enlargement above at every time, but we do not see 
how to do that. 

Yet, proposition~\ref{sufficient condition for maximality : Brownian case} 
in the next subsection shows that that equality $\bc = \bc'$ ensures 
the maximality of $\bc$ under the (strong) additional hypothesis that 
$\bc$ is complementable after $0$. 

The next sufficient condition for the maximality of a poly-adic immersed 
filtration comes from~\cite{Leuridan}. 

\begin{prop}\label{sufficient condition for maximality}
Let $(\bc_n)_{n \le 0}$ be a $(b_n)_{n \le 0}$-adic filtration immersed 
in $(\zc_n)_{n \le 0}$. Let $U$ be a random variable valued in some 
measurable space $(E,\ec)$, such that $\sigma(U) = \bc_0$ and 
$(\ppf_u)_{u \in E}$ a regular version of the conditional probability 
$\ppf$ given $U$. 
If for $U(\ppf)$-almost every $u \in E$, the filtered probability space 
$(\Omega,\ac,\ppf_u,(\zc_n)_{n \le 0})$ is Kolmogorovian, then 
the filtration $(\bc_n)_{n \le 0}$ is maximal in $(\zc_n)_{n \le 0}$. 
\end{prop}

The assumption that $(\zc_n)_{n \le 0}$ is Kolmogorovian under 
almost every conditional probability $\ppf_u$ echoes to the 
alternative terminology of conditional $K$-automorphisms used by 
Thouvenot in~\cite{Thouvenot}. 

\subsection{Complementability and maximality}

In the three contexts (factors of an automorphism of a Lebsgue space, 
poly-adic filtrations immersed in a filtration indexed by $\zzf_-$, 
Brownian filtrations immersed in a Brownian filtration), we get very 
similar results.  

The first one is stated by Ornstein in~\cite{Ornstein} as a direct 
consequence of a lemma stated in~\cite{Ornstein - Weiss}. 

\begin{prop}\label{implication: factor case}
Assume that $T$ with finite entropy. Let $\bc$ be a factor of $T$. If $\bc$
is complementable by some $K$-automorphism, then $\bc$ is maximal.  
\end{prop}

The second one comes from~\cite{Leuridan}.

\begin{prop}\label{implication: poly-adic case}
Let $(\bc_n)_{n \le 0}$ be a $(b_n)_{n \le 0}$-adic filtration immersed 
in $(\zc_n)_{n \le 0}$. If $(\bc_n)_{n \le 0}$ can be complemented by 
some Kolmogorovian filtration, then $(\bc_n)_{n \le 0}$ is maximal
in $(\zc_n)_{n \le 0}$.
\end{prop}

A particular case of the third one (in which the dimension of the 
Brownian filtrations are 1 and 2) can be found in~\cite{Brossard - Leuridan}
or~\cite{Brossard - Emery - Leuridan}. 

\begin{prop}\label{implication: Brownian case}
Let $(\bc_t)_{t \ge 0}$ be a Brownian filtration immersed 
in a Brownian filtration $(\zc_t)_{t \ge 0}$ with larger dimension. 
If $(\bc_t)_{t \ge 0}$ can be complemented by 
some Brownian filtration, then $(\bc_t)_{t \ge 0}$ is maximal
in $(\zc_t)_{t \ge 0}$.
\end{prop}

The proofs of these three statements are rather simple, and present 
some similarities, although they are different. In the next section, we 
provide two different proofs of proposition~\ref{implication: factor case}. 
The first one relies on Ornstein and Weiss' lemma 
(lemma 2 in~\cite{Ornstein - Weiss}). 
The second one is a bit simpler but requires that $T$ has finite entropy, 
and relies on Berg's lemma (lemma 2.3 in~\cite{Berg}). 
We also provide a proof of proposition~\ref{implication: Brownian case}. 
Proposition~\ref{implication: poly-adic case} follows from 
proposition~\ref{sufficient condition for maximality} and 
corollary~\ref{conditionning}.

The converses of the three implications above are false, but providing 
counter-exam\-ples is difficult. 
Ornstein gived in~\cite{Ornstein} an example of maximal but 
non-complementable factor in~\cite{Ornstein}, but the proof 
is difficult to read. Two counter-examples of a maximal but 
non-complementable poly-adic filtration are given in~\cite{Leuridan}. 
In the present paper, we use a third example to construct a maximal but 
non-complementable Brownian filtration. 

In the present paper, we will also use a small refinement of 
proposition~\ref{implication: Brownian case}, using the notion of 
complementability after $0$.

\begin{defi}
Let $(\bc_t)_{t \ge 0}$ be a Brownian filtration immersed in a Brownian 
filtration $(\zc_t)_{t \ge 0}$ with larger dimension. One says that 
$(\bc_t)_{t \ge 0}$ is complementable after $0$ in $(\zc_t)_{t \ge 0}$ if 
there exists some Brownian filtration $\cc$ immersed in $\zc$ and 
independent of $\bc$ such that,  
$$\forall t \ge 0, \zc_t 
= \bigcap_{s > 0} (\bc_t \vee \cc_t \vee \zc_s) \mod \pi.$$
\end{defi}


\begin{prop}\label{sufficient condition for maximality : Brownian case}
Let $(\bc_t)_{t \ge 0}$ be a $d$-dimensional Brownian filtration immersed 
in a Brownian filtration $(\zc_t)_{t \ge 0}$ with larger dimension. 
If $(\bc_t)_{t \le 0}$ is complementable after $0$, then the filtration 
provided by proposition~\ref{slightly larger Brownian filtration}
is the largest $d$-dimensional Brownian filtration immersed in 
$(\zc_t)_{t \le 0}$ and containing $(\bc_t)_{t \le 0}$. In particular, 
$(\bc'_t)_{t \le 0}$ is maximal in $(\zc_t)_{t \le 0}$.
\end{prop}

\section{Conditions for maximality: proofs}
\label{conditions for maximality: proofs}

In this section, we provide proofs of the statements given in 
subsection~\ref{conditions for maximality}, except 
proposition~\ref{slightly larger poly-adic filtration} 
which is proved in~\cite{Leuridan}.

\subsection{Proof of proposition~\ref{conditional Pinsker factor}}

By definition, $\bc'$ is closed under taking complements. 
For every $A$ and $B$ in $\zc$, the partition $\{A \cup B, (A \cup B)^c\}$ 
is less fine that $\{A,A^c\} \vee \{B,B^c\}$, hence 
\begin{eqnarray*}
h(T,\{A \cup B, (A \cup B)^c\}|\bc) 
&\le& h(T,\{A,A^c\}\vee \{B,B^c\}|\bc) \\
&\le& h(T,\{A,A^c\}|\bc) + h(T,\{B,B^c\}|\bc).
\end{eqnarray*}
One deduce that $\bc'$ is closed under finite union.

But $h(T,\{A,A^c\}|\bc)$ depends continuously on $A$
when $\zc$ is endowed with the pseudo-metric defined by 
$\delta(A,B) = \pi(A \triangle B)$ (see proposition~\ref{continuity}), 
so $\bc'$ is a closed subset. Hence, $\bc'$ is a complete $\sigma$-field. 

The equalities $h(T,\{A,A^c\}|\bc) = h(T,\{T^{-1}A,T^{-1}A^c\}|\bc) 
= h(T,\{TA,TA^c\}|\bc)$ for every $A \in \zc$ show that $\bc'$ is a 
factor. 

Moreover, $\bc \subset \bc'$ since for every $B \in \bc$, 
$h(T,\{B,B^c\}|\bc) \le H(\{B,B^c\}|\bc) = 0$.

The sub-additivity of entropy shows that $h(T,\alpha|\bc) = 0$
for every finite partition $\alpha \subset \bc'$. Hence 
$h(T,\bc')-h(T,\bc) = h((T,\bc')|\bc) = 0$. 

Last, let $\ac$ be a factor containing $\bc$ and having the same entropy 
as $\bc$. Then for every $A \in \ac$, 
$$h(T,\{A,A^c\}|\bc) \le h((T,\ac)|\bc) = h(T,\ac)-h(T,\bc) = 0,$$ 
so $\ac \subset \bc'$. The proof is complete. 

\subsection{Proof of proposition~\ref{slightly larger factor}}

The proofs given here are inspired by the proofs of the similar results 
involving (non-conditional) Pinsker factor given 
in~\cite{Cornfeld - Fomin - Sinai}.

For every countable measurable partition $\alpha$ of $(Z,\zc,\pi)$ 
and for every integers $p \le q$, we introduce the notations 
$$\alpha_p^q := \bigvee_{k=p}^{q} T^{-k}\alpha, \quad
\alpha_1^\infty = \bigvee_{k \ge 1}T^{-k}\alpha, \quad 
\bc^\alpha := \bigcap_{n \ge 0} 
\Big(\bc \vee \bigvee_{k \ge n}T^{-k}\alpha \Big).$$
Let us recall that the inclusion 
$$\bc^\alpha \supset \bc \vee \bigcap_{n \ge 0} 
\Big(\bigvee_{k \ge n}T^{-k}\alpha \Big)$$ 
can be strict modulo $\ppf$. We also note that the larger is $n$, 
the smaller is the partition 
$$T^{-n}\alpha_1^\infty = \bigvee_{k \ge n+1}T^{-k}\alpha,$$
so $\bc^\alpha$ is also the intersection of the non-increasing sequence  
$(\bc \vee T^{-n}\alpha_1^\infty)_{n \ge 0}$.

We begin with the following lemma. 

\begin{lemm}\label{conditional entropy given B gamma}
Let $\alpha$ and $\gamma$ be countable measurable partitions of 
$(Z,\zc,\pi)$, with finite entropy. Then
$H(\alpha|\alpha_1^\infty \vee \bc^\gamma) = H(\alpha|\alpha_1^\infty \vee \bc)$. 
\end{lemm}

\begin{proof}
The addition formula for conditional entropy yields for every $n \ge 1$,
$$H(\alpha \vee \cdots \vee T^{n-1}\alpha|\alpha_1^\infty \vee \bc) 
= \sum_{k=0}^{n-1}H(T^k\alpha|T^k\alpha_1^\infty \vee \bc)
= nH(\alpha|\alpha_1^\infty \vee \bc).$$
Replacing $\alpha$ with $\alpha \vee \gamma$ gives 
$$U_n := 
H(\alpha_{-n+1}^0 \vee \gamma_{-n+1}^0|\alpha_1^\infty \vee \bc \vee \gamma_1^\infty) 
= nH(\alpha \vee \gamma|\alpha_1^\infty \vee \bc \vee \gamma_1^\infty).$$
Set
$$V_n := H(\alpha_{-n+1}^0 \vee \gamma_{-n+1}^0|\alpha_1^\infty \vee \bc)
\text{ and } W_n := H(\alpha_{-n+1}^0 \vee \gamma_{-n+1}^0|\bc).$$ 
Since $U_n \le V_n \le W_n$ and
$$\lim_n W_n/n
= h((T,\alpha \vee \gamma)|\bc)
= H(\alpha \vee \gamma|\alpha_1^\infty \vee \bc \vee \gamma_1^\infty),$$ 
we get $\lim_n U_n/n = \lim_n V_n/n$.
But 
$$U_n = H(\alpha_{-n+1}^0|\alpha_1^\infty \vee \bc \vee \gamma_1^\infty)
+ H(\alpha_{-n+1}^0 \vee \gamma_{-n+1}^0|T^n\alpha_1^\infty \vee \bc \vee \gamma_1^\infty),$$ 
$$V_n = H(\alpha_{-n+1}^0|\alpha_1^\infty \vee \bc)
+ H(\alpha_{-n+1}^0 \vee \gamma_{-n+1}^0|T^n\alpha_1^\infty \vee \bc).$$ 
In these two expressions, each term in $U_n$ is less or equal that the 
corresponding term in $V_n$. 
Since $H(\alpha_{-n+1}^0|\alpha_1^\infty \vee \bc) 
= n H(\alpha|\alpha_1^\infty \vee \bc)$ 
we derive that 
$$\lim_n n^{-1} H(\alpha_{-n+1}^0|\alpha_1^\infty \vee \bc \vee \gamma_1^\infty)
= H(\alpha|\alpha_1^\infty \vee \bc).$$
Furthermore, we note that   
$$\alpha_1^\infty \vee \bc^\gamma \subset \bigcap_{n \ge 0} 
\Big(\alpha_1^\infty \vee \bc \vee T^{-n}\gamma_1^\infty \Big),$$
so proposition~\ref{conditional entropy and monotone sequence of sigma-fields},
Ces\`aro's lemma and addition formula for conditional entropy yield
\begin{eqnarray*}
H(\alpha|\alpha_1^\infty \vee \bc^\gamma) 
&\ge& \lim_n H(\alpha|\alpha_1^\infty \vee \bc \vee T^{-n}\gamma_1^\infty) \\
&=& \lim_n H(T^n\alpha|T^n\alpha_1^\infty \vee \bc \vee \gamma_1^\infty) \\
&=& \lim_n n^{-1}\sum_{k=0}^{n-1}H(T^k\alpha|T^k\alpha_1^\infty \vee \bc \vee \gamma_1^\infty) \\
&=& \lim_n n^{-1}H(\alpha_{-n+1}^0|\alpha_1^\infty \vee \bc \vee \gamma_1^\infty) \\
&=& H(\alpha|\alpha_1^\infty \vee \bc).
\end{eqnarray*}
But the reverse inequality follows from the inclusion 
$\alpha_1^\infty \vee \bc \subset \alpha_1^\infty \vee \bc^\gamma$. 
Hence the equality holds. 
\end{proof}

Lemma~\ref{conditional entropy given B gamma} yields one inclusion
in the equality of proposition~\ref{slightly larger factor}, thanks 
to the next corollary.

\begin{coro}
For every countable measurable partition $\gamma$ 
of $(Z,\zc,\pi)$, with finite entropy, $\overline{\bc^\gamma} \subset \bc'$. 
\end{coro}

\begin{proof}
Since $\bc'$ is complete, one only needs to check that 
$\bc^\gamma \subset \bc'$.
Let $A \in \bc^\gamma$ and $\alpha = \{A,A^c\}$. 
Lemma~\ref{conditional entropy given B gamma} yields
$h((T,\alpha)|\bc) = H(\alpha|\alpha_1^\infty \vee \bc) 
= H(\alpha|\alpha_1^\infty \vee \bc^\gamma) = 0$, so $A \in \bc'$. 
The inclusion follows.
\end{proof}

Lemma~\ref{conditional entropy given B gamma} will also help us to prove 
the next useful lemma. 

\begin{lemm}\label{B alpha gamma}
Let $\alpha$ and $\gamma$ be countable measurable partitions 
of $(Z,\zc,\pi)$, with finite entropy. Let $N \ge 0$ and 
$\eta$ be a finite partition less fine that 
$\gamma_{-N}^N = \bigvee_{k=-N}^N T^{-k}\gamma$. Then  
$$H(\eta|(\bc^\alpha)^\gamma) = H(\eta|\bc^\gamma \vee \bc^\alpha) 
= H(\eta|\bc^\gamma).$$
\end{lemm}

\begin{proof}
For every $n \ge 0$, $(\gamma_{-n}^n)_1^\infty = T^n\gamma_1^\infty$, so 
$H(\gamma_{-n}^n|T^n\gamma_1^\infty \vee \bc^\alpha) =
H(\gamma_{-n}^n|T^n\gamma_1^\infty \vee \bc)$ 
by lemma~\ref{conditional entropy given B gamma}.
When $n \ge N$, $\gamma_{-n}^n = \eta \vee \gamma_{-n}^n$, so
$$H(\gamma_{-n}^n|T^n\gamma_1^\infty \vee \bc^\alpha) =
H(\eta|T^n\gamma_1^\infty \vee \bc^\alpha) + 
H(\gamma_{-n}^n|T^n\gamma_1^\infty \vee \bc^\alpha \vee \eta),$$ 
$$H(\gamma_{-n}^n|T^n\gamma_1^\infty \vee \bc) =
H(\eta|T^n\gamma_1^\infty \vee \bc) + 
H(\gamma_{-n}^n|T^n\gamma_1^\infty \vee \bc \vee \eta).$$ 
Since $H(\beta|\ac \vee \bc^\alpha) \le H(\beta|\ac \vee \bc)$ 
for every countable measurable partition $\beta$ and every 
$\sigma$-field $\ac \subset \zc$, we get 
$H(\eta|T^n\gamma_1^\infty \vee \bc^\alpha) = H(\eta|T^n\gamma_1^\infty \vee \bc)$.
Letting $n$ go to infinity yields 
$H(\eta|(\bc^\alpha)^\gamma) = H(\eta|\bc^\gamma)$. Since
$\bc^\gamma \subset \bc^\gamma \vee \bc^\alpha \subset (\bc^\alpha)^\gamma$,
the result follows.
\end{proof}

\begin{coro}
Assume that $T$ has finite entropy and admits a generator $\gamma$. 
Then for every countable measurable partition $\alpha$ ,
$\bc^\alpha \subset \bc^\gamma = (\bc^\gamma)^\gamma \mod \pi$. 
\end{coro}

\begin{proof}
The collection of all $A \in \zc$ such that 
$$H(\{A,A^c\}|\bc^\gamma) = H(\{A,A^c\}|\bc^\gamma \vee \bc^\alpha) =  
H(\{A,A^c\}|(\bc^\gamma)^\gamma)$$
is a closed subset for the pseudo-metric defined by 
$\delta(A,B) = \pi(A \triangle B)$, 
and contains the algebra $\bigcup_{N \in \nnf} \sigma(\gamma_{-N}^N)$ 
by lemma~\ref{B alpha gamma} applied once to $(\alpha,\gamma)$
and once to $(\gamma,\gamma)$. Therefore, these collection equals 
the whole $\sigma$-field $\zc$ itself. In particular, 
$H(\{A,A^c\}|\bc^\gamma) = 0$ whenever $A \in \bc^\gamma \vee \bc^\alpha$ or 
$A \in (\bc^\gamma)^\gamma$. 
Hence $\bc^\gamma \vee \bc^\alpha$ and 
$(\bc^\gamma)^\gamma$ are contained in $\bc^\gamma$ modulo the null sets. 
The result follows. 
\end{proof}

We now achieve the proof of proposition~\ref{slightly larger factor}. 
Assume that $T$ has finite entropy and that $\gamma$ is a generator of $T$. 
We have to prove that $\bc' \subset \overline{\bc^\gamma}$.
So let $A \in \bc'$ and $\alpha = \{A,A^c\}$. For every $n \ge 0$, set 
$$\dc_n := \sigma\big(\bigvee_{k \ge n} T^{-k}\alpha\big).$$ 
Since, $\dc_1=\alpha_1^\infty$, the equality
$H(\alpha|\alpha_1^\infty \vee \bc) = h(T,\alpha|\bc) = 0$ yields
$\alpha \subset \dc_1 \vee \bc \mod \pi$, so 
$\dc_0 = \dc_1 \vee \bc \mod \pi$. By applying $T^{-n}$, we get more generally 
$\dc_n = \dc_{n+1} \vee \bc \mod \pi$. By induction, 
$\dc_0 = \dc_n \vee \bc \mod \pi$ for every $n \ge 0$. 
Hence $\dc_0 = \bc^\alpha \subset \bc^\gamma \mod \pi$, 
thanks to the last corollary, so $A \in \overline{\bc^\gamma}$. 
We are done.

\subsection{Proof of proposition~\ref{slightly larger Brownian filtration}}

Fix a $d$-dimensional Brownian motion $B$ generating the filtration 
$\bc$ modulo the null sets. 

Let $t>\ep>0$. Since $\bc$ is immersed in $\zc$, the Brownian motion 
$B^{(\ep)} := B_{\ep+\cdot}-B_\ep$ is independent of $\zc_\ep$, which is the 
terminal $\sigma$-field of the filtration $(\bc_\ep \vee \zc_s)_{s \in [0,\ep]}$. 
Therefore
\begin{eqnarray*}
\bc'_t 
&=& \bigcap_{s \in ]0,\ep]} 
\big( \sigma((B^{(\ep)}_r)_{r \in [0,t-\ep]}) \vee \bc_\ep \vee \zc_s \big) \\
&=& \sigma((B^{(\ep)}_r)_{r \in [0,t-\ep]}) 
\vee \bigcap_{s \in ]0,\ep]} \big( \bc_\ep \vee \zc_s \big) \\
&=& \sigma((B^{(\ep)}_r)_{r \in [0,t-\ep]}) \vee \bc'_\ep \mod \ppf.
\end{eqnarray*}
Hence, the filtration $\bc'$ has independent increments after $\ep$, 
provided by the increments of $B$ after $\ep$, so $\bc'$ is Brownian 
after $0$. 

Moreover, since $B^{(\ep)}$ is independent of $\zc_\ep$, the equality 
modulo $\ppf$ above shows that for every $t>\ep>0$, 
$\bc'_t$ and $\zc_\ep$ are independent 
conditionally on $\bc'_\ep$. This conditional independence still holds 
when $\ep=0$, since $\zc_0$ is trivial. Thus $\bc'$ is immersed in $\zc$. 
By proposition~\ref{Brownian after $0$ and immersed in a Brownian filtration}
(or corollary 1 in~\cite{Emery}), $\bc'$ is a $d$-dimensional 
Brownian filtration. 

For every $t>0$, $\bc_t \subset \bc'_t \subset \bc_t \vee \zc_t 
= \zc_t \mod \pi$. These inclusions modulo $\ppf$ still hold when 
$t=0$ since $\bc_0$ and $\zc_{0+} = \bigcap_{s>0}\zc_s$ are trivial. 
Since $\bc$ is immersed in $\zc$, we deduce that $\bc$ is immersed in $\bc'$. 

Last, let $t \ge 0$. For every $n \ge 1$, 
$$\bc''_t \subset \overline{\bc'_t \vee \zc_{1/n}}
\subset \overline{\overline{(\bc_t \vee \zc_{1/n})} \vee \zc_{1/n}} 
= \overline{\bc_t \vee \zc_{1/n}}.$$ 
If $A \in \bc''_t$, then for each $n \ge 1$, one can find 
$B_n \in \bc_t \vee \zc_{1/n}$ such that $\ppf(A \triangle B_n) = 0$; 
hence $A \in \bc'_t$ since $B := \limsup_n B_n$ belongs to 
$\bigcap_{n \ge 1}(\bc_t \vee \zc_{1/n})$ and $\ppf(A \triangle B) = 0$. 
The equality $\bc''_t=\bc'_t$ follows.

\subsection{Proof of the statements of example~\ref{Levy}}

Assume that $\xc$ is maximal in $\zc$. 
Since $\bc$ is immersed $\bc'$, we have only to check the inclusion 
$\bc'_\infty \subset \bc_\infty \mod \ppf$. 
The maximality of $\xc$ in $\zc$ yields 
$$\bc'_\infty \subset \bigcap_{s>0} (\xc_\infty \vee \zc_s) 
= \xc_\infty \mod \ppf,$$
so one only needs to check that 
$\eef[h(X)|\bc'_\infty] = \eef[h(X)|\bc_\infty]$ almost surely 
for every real bounded measurable functional $h$ defined on the space 
$\cc(\rrf_+)$ of all continous functions from $\rrf_+$ to $\rrf$. 
Since the topology of uniform convergence on compact subsets 
on the space $\cc(\rrf_+)$ is metrizable, it is sufficient to 
check the equality when $h$ is continuous. In this case, the 
random variable $h(X)$ is the limit in $L^1(\ppf)$ of the sequence 
$(h(X^{(n)}))_{n \ge 1}$, where $T_n$ denotes the first zero of $X$ 
after time $1/n$, and $X^{(n)}_t = X_{T_n+t}$ for every $t \ge 0$. 
Since $B$ generates the same filtration as $|X|$ up to null sets, 
$\bc_\infty \vee \zc_{T_n} = \sigma(|X^{(n)}|) \vee \zc_{T_n} \mod \ppf$. 
But $X^{(n)}$ is independent of $\zc_{T_n}$ since $\xc$ is immersed in $\zc$,
so $$\eef[h(X^{(n)})|\bc_\infty \vee \zc_{T_n}] 
= \eef[h(X^{(n)})|\sigma(|X^{(n)}|) \vee \zc_{T_n}]
= \eef[h(X^{(n)})|\sigma(|X^{(n)}|)] \text{ a.s.}.$$
But $\sigma(|X^{(n)}|) \subset \bc_\infty \subset \bc'_\infty \subset \bc_\infty \vee \zc_{1/n} \subset \bc_\infty \vee \zc_{T_n}$, so 
$$\eef[h(X^{(n)})|\bc'_\infty] = \eef[h(X^{(n)})|\bc_\infty].$$
The statements follow.

\section{Complementability implies maximality: proofs}
\label{complementability implies maximality: proofs}

\subsection{Key lemma for factors of a Lebesgue automorphism}

Proposition~\ref{implication: factor case} follows from the next lemma, 
which can derived from lemma in~\cite{Ornstein - Weiss} or from 
lemma~2.3 in~\cite{Berg}.

\begin{lemm}\label{Ornstein - Weiss}
Let $\ac,\bc,\cc$ be three factors of $T$. Assume that: 
\begin{enumerate}
\item $\ac \supset \bc$;
\item $h(T,\ac) = h(T,\bc)<+\infty$;
\item $(T,\cc)$ is a $K$-automorphism and has finite entropy. 
\item $\bc$ and $\cc$ are independent. 
\end{enumerate}
Then $\ac$ and $\cc$ are independent. 
\end{lemm}

First, we show how to deduce proposition~\ref{implication: factor case}
from lemma~\ref{Ornstein - Weiss}.

\begin{proof} (Proof of proposition~\ref{implication: factor case})
Let $\cc$ be an independent complement of $\bc$ having the property $K$. 
Let $\ac$ be a factor of $T$ such that $\ac \supset \bc$ and 
$h(T,\ac) = h(T,\bc)$. Then lemma~\ref{Ornstein - Weiss} yields that 
$\ac$ and $\cc$ are independent. But 
$\zc = \bc \vee \cc \subset \ac \vee \cc \subset \zc$, so 
$\ac \vee \cc = \bc \vee \cc$. Hence $\ac = \bc$ by the next lemma. 
\end{proof}

We have just used the follwing general statement, 
which will also help us in the context of Brownian filtrations.    

\begin{lemm}\label{simplifying by C}
Let $\ac,\bc,\cc$ be three sub-$\sigma$-fields of any probability space
$(Z,\zc,\pi)$ such that 
\begin{itemize}
\item $\ac \supset \bc$;
\item $\ac$ and $\cc$ are independent;
\item $\ac \vee \cc = \bc \vee \cc$. 
\end{itemize}
Then $\ac = \bc \mod \pi$. 
\end{lemm}

\begin{proof} 
Let $A \in \ac$. Then $\sigma(A) \vee \bc \subset \ac$, so 
$\sigma(A) \vee \bc$ is independent of $\cc$, and 
$$\pi[A|\bc] = \pi[A|\bc  \vee \cc] = \pi[A|\ac  \vee \cc] 
= \one_A~\pi\mathrm{-almost~surely}.$$
Hence $A \in \bc \mod \pi$. 
\end{proof}

We now give two different proofs of lemma~\ref{Ornstein - Weiss}. 
The second one relies on Pinsker's formula and is a bit simpler.

\subsection{Proof of lemma~\ref{Ornstein - Weiss}}

The proof below is a reformulation of the proof given 
in~\cite{Ornstein - Weiss}. 

Assume that the assumptions hold. Let $\alpha$, $\beta$, $\gamma$ be 
countable partitions generating $(T,\ac)$, $(T,\bc)$, $(T,\cc)$, respectively. 
Since $(T,\ac)$, $(T,\bc)$, $(T,\cc)$ have finite entropy, 
$\alpha$, $\beta$, $\gamma$ have also finite entropy. 
Given $n \ge 1$, set $\alpha_0^{n-1} = \alpha \vee \cdots \vee T^{-(n-1)}\alpha$, 
$\beta_0^{n-1} = \beta \vee \cdots \vee T^{-(n-1)}\beta$,
$$\cc_n = \sigma \Big( \bigvee_{q \in \zzf} T^{-qn}\gamma \Big) \text{ and } 
\dc_n = \sigma \Big( \bigvee_{k \ge n} T^{-k}\gamma \Big).$$
Then $\cc_n$ is a factor of $T^n$, $\alpha_0^{n-1} \vee \gamma$ is a 
generator of  $(T^n,\ac \vee \cc_n)$ whereas 
$\beta_0^{n-1} \vee \gamma$ is a generator of  
$(T^n,\bc \vee \cc_n)$.  

Therefore, on the one hand,
\begin{eqnarray*}
h(T^n,\ac \vee \cc_n) 
&=& H \Big(\alpha_0^{n-1} \vee \gamma \Big| 
\bigvee_{q \ge 1} T^{-qn}(\alpha_0^{n-1} \vee \gamma)\Big)\\
&=& H \Big(\alpha_0^{n-1} \vee \gamma \Big| 
\bigvee_{k \ge n} T^{-k}\alpha \vee \bigvee_{q \ge 1} T^{-qn}\gamma \Big)\\
&=& H \Big(\alpha_0^{n-1} \Big| 
\bigvee_{k \ge n} T^{-k}\alpha \vee \bigvee_{q \ge 1} T^{-qn}\gamma \Big)
+ H \Big(\gamma \Big| 
\bigvee_{k \ge 0} T^{-k}\alpha \vee \bigvee_{q \ge 1} T^{-qn}\gamma \Big)\\
&\le& H \Big(\alpha_0^{n-1} \Big| 
\bigvee_{k \ge n} T^{-k}\alpha \Big) + H(\gamma|\alpha)\\
&=& h(T^n,\ac) + H(\gamma|\alpha) = nh(T,\ac) + H(\gamma|\alpha).
\end{eqnarray*}

On the other hand, by independence of $\bc$ and $\cc$, 
\begin{eqnarray*}
h(T^n,\bc \vee \cc_n) 
&=& H \Big(\beta_0^{n-1} \Big| 
\bigvee_{k \ge n} T^{-k}\beta \vee \bigvee_{q \ge 1} T^{-qn}\gamma \Big)
+ H \Big(\gamma \Big| 
\bigvee_{k \ge 0} T^{-k}\beta \vee \bigvee_{q \ge 1} T^{-qn}\gamma \Big)\\
&=& H \Big(\beta_0^{n-1} \Big| \bigvee_{k \ge n} T^{-k}\beta \big) 
+ H \Big(\gamma|\bigvee_{q \ge 1} T^{-qn}\gamma \Big)\\
&\ge& h(T^n,\bc) + H(\gamma|\dc_n) = nh(T,\bc) + H(\gamma|\dc_n).
\end{eqnarray*}
But $h(T^n,\bc \vee \cc_n) \le h(T^n,\ac \vee \cc_n)$ since $\bc \subset \ac$. 
Putting things together and using the assumption 
$h(T,\ac) = h(T,\bc)<+\infty$ yields $H(\gamma|\dc_n) \le H(\gamma|\alpha)$. 

But $(\dc_n)_{n \ge 1}$ is a decreasing sequence of $\sigma$-fields with 
trivial intersection since $(T,\cc)$ has the property $K$, so 
$H(\gamma|\dc_n) \to H(\gamma)$ as $n \to +\infty$. 
Hence, $H(\gamma) \le H(\gamma|\alpha)$, so $\alpha$ and 
$\gamma$ are independent. This conclusion is preserved if one replaces the 
generators $\alpha$ and $\gamma$ by the supremum of $T^{-k}\alpha$ and 
$T^{-k}\gamma$ over all $k \in \odc -n,n \fdc$. Letting $n$ go to  
infinity yields the independence of $\ac$ and $\cc$.

\subsection{Alternative proof of lemma~\ref{Ornstein - Weiss}}

The inclusion $\ac \supset \bc$ and the independence of $\bc$ and $\cc$ yield 
$$h(T,\ac) + h(T,\cc) \ge h(T,\ac \vee \cc) 
\ge h(T,\bc \vee \cc) = h(T,\bc) + h(T,\cc).$$
But $h(T,\ac) = h(T,\bc)$, hence $h(T,\ac \vee \cc) = h(T,\ac) + h(T,\cc)$. 
Since $(T,\cc)$ is a $K$-automorphism with finite entropy, 
Berg's lemma below shows that $\ac$ and $\cc$ independent. 

\begin{lemm}(lemma 2.3 in~\cite{Berg})
Let $\ac$ and $\cc$ be two factors of the dynamical system $(Z,\zc,\pi,T)$, 
such that $h(T,\ac \vee \cc) = h(T,\ac) + h(T,\cc) < +\infty$ and 
$(T,\cc)$ is a $K$-automorphism. Then $\ac$ and $\cc$ independent.
\end{lemm}

\begin{proof}
Let $\alpha$ and $\gamma$ be countable generating 
partitions of $(T,\ac)$ and $(T,\cc)$, respectively. Set 
$$\alpha_1^\infty = \bigvee_{k \ge 1} T^{-k}\alpha 
\text{ and } \gamma_1^\infty = \bigvee_{k \ge 1} T^{-k}\gamma$$
Then $h(T,\ac) = H(\alpha|\alpha_1^\infty)$, $h(T,\cc) = H(\gamma|\gamma_1^\infty)$ and 
$h(T,\ac \vee \cc) = H(\alpha \vee \gamma|\alpha_1^\infty \vee \gamma_1^\infty)$.  
But Pinsker's formula (proposition~\ref{Pinsker's formula} in 
section~\ref{Annex} or theorem 6.3 in~\cite{Parry}) gives 
$$H(\alpha \vee \gamma|\alpha_1^\infty \vee \gamma_1^\infty) 
= H(\alpha|\alpha_1^\infty) + H(\gamma|\ac \vee \gamma_1^\infty).$$ 
So the assumption $h(T,\ac \vee \cc) = h(T,\ac) + h(T,\cc) < +\infty$ 
yields $H(\gamma|\gamma_1^\infty) = H(\gamma|\ac \vee \gamma_1^\infty)$.  

For any partition $\delta \subset \ac$ with finite entropy, we derive
$H(\gamma|\delta \vee \gamma_1^\infty) = H(\gamma|\gamma_1^\infty)$, so
$$H(\delta \vee \gamma|\gamma_1^\infty) 
= H(\delta|\gamma_1^\infty) + H(\gamma|\delta \vee \gamma_1^\infty)
= H(\delta|\gamma_1^\infty) + H(\gamma|\gamma_1^\infty).$$
But we have also
$$H(\delta \vee \gamma|\gamma_1^\infty) 
= H(\gamma|\gamma_1^\infty) + H(\delta|\gamma \vee \gamma_1^\infty).$$
Hence $H(\delta|\gamma_1^\infty) = H(\delta|\gamma \vee \gamma_1^\infty)$.

Let $m \ge 0$ and $n$ be integers. Applying the last equality to
$\delta:=\bigvee_{|k| \le m} T^{n-k}\alpha$ yields
$$H\Big(\bigvee_{|k| \le m} T^{n-k}\alpha~\Big| 
\bigvee_{k \ge 1} T^{-k}\gamma \Big) 
= H\Big(\bigvee_{|k| \le m} T^{n-k}\alpha~\Big| 
\bigvee_{k \ge 0} T^{-k}\gamma \Big).$$
Since $T$ preserves $\pi$, this is equivalent to
$$H\Big(\bigvee_{|k| \le m} T^{-k}\alpha~\Big| 
\bigvee_{k \ge n+1} T^{-k}\gamma \Big) 
= H\Big(\bigvee_{|k| \le m} T^{-k}\alpha~\Big| 
\bigvee_{k \ge n} T^{-k}\gamma \Big).$$
As a result, the entropy above does not depend on $n$. 
Letting $n$ go to $-\infty$ and to $+\infty$, and using 
the fact that $(T,\cc)$ is a $K$-automorphism, we get 
at the limit
$$H\Big(\bigvee_{|k| \le m} T^{-k}\alpha~\Big| \cc \Big) 
= H\Big(\bigvee_{|k| \le m} T^{-k}\alpha \Big),$$
so the partition $\bigvee_{|k| \le m}T^{-k}\alpha$ is independent of $\cc$. 
Letting $m$ go to $+\infty$ yields the independence of 
$\ac$ and $\cc$.
\end{proof}

\subsection{Proof in the context of Brownian filtrations}

The proof of proposition~\ref{implication: Brownian case} below may 
look suspiciously simple, but actually, it relies on non-trivial theorems 
of stochastic integration, namely the predictable representation property 
and the bracket characterization of multi-dimensional Brownian motions 
among local martingales. 
The immersion of a filtration into another one is a strong property, 
as shown for example by the characterizations for a Brownian filtration 
recalled in the introduction 
(proposition~\ref{immersion of a Brownian filtration}). 
The key step is very similar to lemma~\ref{Ornstein - Weiss}.

\begin{lemm}\label{key step, Brownian case}
Let $A,B,C$ be three Brownian motions in some filtration $\zc$.  
Assume that: 
\begin{enumerate}
\item $\sigma(A) \supset \sigma(B)$;
\item $A$ and $B$ have the same finite dimension;
\item $B$ and $C$ are independent. 
\end{enumerate}
Then $A$ and $C$ are independent. 
\end{lemm}

\begin{proof}
Call $p$ the dimension of $A$ and $B$ and $q$ the (possibly infinite) 
dimension of $C$. 
Since $B$ is a Brownian motion in $\zc$ and its own filtration, 
it is also a Brownian motion in the intermediate filtration $\fc^A$. 
Hence, one can find an $\fc^A$-predictable process $H$ with values 
in the group of all orthogonal $p \times p$ matrices such that  
$$B = \int_0^\cdot H_s \d A_s.$$
Since $H_s^\top H_s = I_p$ for every $s \ge 0$ (where $H_s^\top$ denotes 
the transpose of $H_s$), we have also 
$$A = \int_0^\cdot H_s^\top \d B_s.$$
Looking at the components, we get for every $i \in \odc 1,p \fdc$,
$$A^{(i)} = \sum_{j=1}^p \int_0^\cdot H_s(j,i) \d B^{(j)}_s.$$
For every $i \in \odc 1,p \fdc$ and $k \in \odc 1,q \fdc$, we get 
$$\langle A^{(i)},C^{(k)} \rangle = \sum_{j=1}^p \int_0^\cdot H_s(j,i) 
\d\langle B^{(j)},C^{(k)} \rangle_s = 0,$$
since $\langle B^{(j)},C^{(k)} \rangle = 0$ by independence 
of $B$ et $C$. 
We derive that $(A,C)$ is a $p+q$-dimensional Brownian motion in $\zc$, 
so $A$ and $C$ are independent.  
\end{proof}

Deducing proposition~\ref{implication: Brownian case} 
from the last lemma involves almost the same arguments as deducing 
proposition~\ref{implication: factor case}
from lemma~\ref{Ornstein - Weiss}.

\begin{proof}
Proof of proposition~\ref{implication: Brownian case}. 
Let $\zc$ be a finite Brownian filtration, and $\ac,\bc,\cc$ be 
three Brownian filtrations in $\zc$ such that $\ac_t \supset \bc_t$ 
for every $t \ge 0$, $\ac$ and $\bc$ have the same dimension, 
and $\cc$ is an independent complement of $\bc$ in $\zc$. 

Let $A,B,C$ be Brownian motions generating $\ac,\bc,\cc$
modulo the null events. Then 
lemma~\ref{key step, Brownian case} applies, so $A$ is independent 
of $C$. Then lemma~\ref{simplifying by C} applies, so 
$\sigma(A) = \sigma(B) \mod \ppf$.  
But $\fc^B$ is immersed in $\fc^A$. Since the final 
$\sigma$-fields $\fc^A_\infty=\sigma(A)$ and $\fc^B_\infty=\sigma(B)$ 
coincide almost surely, we get $\fc^A_t = \fc^B_t \mod \ppf$ 
for every $t \ge 0$ by lemma~\ref{immersion and final sigma-field}.
\end{proof}

We now prove 
proposition~\ref{sufficient condition for maximality : Brownian case}.

\begin{proof}
Let $\cc$ be a complement of $\bc$ after $0$, 
and $\ac$ be a $d$-dimensional Brownian filtration immersed in $\zc$ 
and containing $\bc$.  
Let $A$, $B$, $C$ be Brownian motions in $\zc$ 
generating $\ac$, $\bc$, $\cc$ respectively modulo the null events. 
Since $\ac$ and $\bc'$ are immersed in $\zc$, it is sufficient 
to prove the inclusion $\ac_\infty \subset \bc'_\infty$. 
Hence, given $s>0$, we have to check that 
$\ac_\infty \subset \bc_\infty \vee \zc_s \mod \ppf$. 

By lemma~\ref{key step, Brownian case}, we know that $A$ and $C$ 
are independent Brownian motions in $\fc^Z$. 
Thus $\zc_s$, $A_{s+\cdot}-A_s$ and $C_{s+\cdot}-C_s$ are independent. 
Let 
$$\tilde{\ac} := \ac_\infty \vee \zc_s 
= \sigma(A_{s+\cdot}-A_s) \vee \zc_s \mod \ppf,$$ 
$$\tilde{\bc} := \bc_\infty \vee \zc_s 
= \sigma(B_{s+\cdot}-B_s) \vee \zc_s \mod \ppf,$$
$$\tilde{\cc} := \sigma(C_{s+\cdot}-C_s).$$
Then $\tilde{\ac} \supset \tilde{\bc}$, $\tilde{\bc}$ and $\tilde{\cc}$ 
are independent, and 
$\tilde{\ac} \vee \tilde{\cc} = \tilde{\bc} \vee \tilde{\cc} \mod \ppf$, since 
$$\zc_\infty \supset \tilde{\ac} \vee \tilde{\cc} \supset 
\tilde{\bc} \vee \tilde{\cc} = \bc_\infty \vee \cc_\infty \vee \zc_s 
= \zc_\infty \mod \ppf.$$
Hence lemma~\ref{simplifying by C} applies, so 
$\ac_\infty \subset \tilde{\ac} = \tilde{\bc} = \bc_\infty \vee \zc_s \mod \ppf$. 
\end{proof}

\section{A maximal but not-complementable Brownian filtration}
\label{maximal but non-complementable Brownian filtration}

\subsection{A maximal but 
not-complementable filtration in a dyadic product-type filtration}

This subsection is devoted to the proof of the following lemma, 
which is the first step of the construction of a maximal 
but not-complementable Brownian filtration.

\begin{lemm}\label{dyadic}
One can construct : 
\begin{itemize}
\item a dyadic product-type filtration $(\zc_n)_{n \le 0}$,
\item a poly-adic filtration $(\uc_n)_{n \le 0}$ immersed in $(\zc_n)_{n \le 0}$,
\item a random variable $U$ with values in some Polish 
space $(E,\ec)$ and generating $\uc_0$,
\end{itemize} 
such that for $U(\ppf)$-almost every $u \in E$, 
$(\zc_n)_{n \le 0}$ is Kolmogorovian but not product-type under 
$\ppf_u = \ppf[\cdot|U=u]$. 
Therefore, the filtration $(\uc_n)_{n \le 0}$ is maximal 
but non complementable in $(\zc_n)_{n \le 0}$. 
\end{lemm}

\begin{proof}
We begin with a variant of an example given 
in~\cite{Ceillier - Leuridan}, 
which was itself inspired from an unpublished note of 
Tsirelson~\cite{Tsirelson}. 

For every $n \le 0$, call $K_n$ the finite field with $q_n = 2^{2^{|n|}}$ 
elements. Start with a sequence of independent random variables 
$(Z_n)_{n \le 0}$ such that for every $n \le 0$, $Z_{2n} = (X_n,Y_n)$ 
is uniform on $K_n^4 \times K_n^4$ and $Z_{2n-1} = B_n$ is uniform on $K_n^4$. 
By construction, the filtration $(\fc^Z_n)_{n \le 0}$ is product-type 
and $(r_n)_{n \le 0}$-adic, with $r_{2n-1} = q_n^4$ and $r_{2n} = q_n^8$ 
for every $n \le 0$. 

Since $|K_{n-1}| = 2^{2^{|n|+1}} = |K_n|^2$, one can fix a bijection between
$K_{n-1}^4 \times K_{n-1}^4$ and the set $\mc_4(K_n)$ of all $4 \times 4$ 
matrices with entries in $K_n$. Call $A_n$ the uniform random variable 
on $\mc_4(K_n)$ corresponding to $Z_{2n-2}$ through this bijection, and 
set $U_{2n-1} = 0$ and $U_{2n} = Y_n - A_nX_n - B_n$.     

For every $n \le 0$, $(X_n,Y_n)$ is independent of 
$\fc^Z_{2n-1}$ and uniform on $K_n^4 \times K_n^4$. Since the random map   
$(x,y) \mapsto (x,y-A_nx-B_n)$ from $K_n^4 \times K_n^4$ to itself 
is $\fc^Z_{2n-1}$-measurable and bijective, $(X_n,U_{2n})$ is also 
independent of $\fc^Z_{2n-1}$ and uniform on $K_n^4 \times K_n^4$ and is 
still an innovation at time $2n$ of the filtration $\fc^Z$.
Therefore, the filtration $(\fc^U_n)_{n \le 0}$ is immersed in 
$(\fc^Z_n)_{n \le 0}$, product-type and $(r_n/q_n^4)_{n \le 0}$-adic. 

As the integers $(r_n)_{n \le 0}$ and $(r_n/q_n^4)_{n \le 0}$ 
are powers of $2$, one can interpolate the filtrations 
$(\fc^Z_n)_{n \le 0}$ and $(\fc^U_n)_{n \le 0}$ to get two 
filtrations $\zc = (\zc_n)_{n \le 0}$ and 
$\uc = (\uc_n)_{n \le 0}$ such that: 
\begin{itemize}
\item $\zc$ is a dyadic product-type filtration, 
\item $\uc$ is poly-adic and immersed in $\zc$,
\item $(\fc^Z_n)_{n \le 0} = (\zc_{t_n})_{n \le 0}$ and 
$(\fc^U_n)_{n \le 0} = (\uc_{t_n})_{n \le 0}$ for some sequence 
$0 = t_0 > t_{-1} > t_{-2} > \ldots$ of integers. 
\end{itemize}
To get such filtrations, it suffices to split the 
random variables $(B_n)_{n \le 0}$, $(X_n)_{n \le 0}$ 
and $(U_{2n})_{n \le 0}$ into independent random bits. 
The bits coming from the $B_n$ and the $X_n$ 
provide innovations of the filtration $\zc$ at the times when the 
filtration $\zc$ only increases. The bits coming from the $U_{2n}$ 
provide innovations of the two filtrations $\zc$ and $\uc$ 
at the times when both of them increase. 

The random variable $U = (U_n)_{n \le 0}$ generates $\uc_0 = \fc^U_0$. 
Let us check that for $U(\ppf)$-almost every $u \in E$, 
$(\zc_n)$ is Kolmogorovian but not product-type under 
$\ppf_u = \ppf[\cdot|U=u]$. 
By corollary~\ref{conditionning}   
and proposition~\ref{sufficient condition for maximality} 
of the present paper
(propositions 3,4 and corollary 9 of \cite{Leuridan}), 
the last two statements will follow. 

First, we note that for every $n \le 0$, 
$\fc^Z_{2n} = \fc^U_{2n} \vee \fc^{B,X}_{n}$, where 
$\fc^U_{2n}$ and $\fc^{B,X}_{n}$ are independent. 
By corollary~\ref{conditionning}, for $U(\ppf)$-almost every $u \in E$,
the filtration $(\fc^Z_{2n})_{n \le 0}$ seen under 
$\ppf_u$ is isomorphic to the filtration 
$(\fc^{B,X}_n)_{n \le 0}$ seen under $\ppf$, which is product-type. 
Thus, $\zc_{-\infty} = \fc^Z_{-\infty}$ is trivial under 
$\ppf_u$. 

To show that $(\zc_n)$ is not product-type under $\ppf_u$, 
it suffices to show that the extracted filtration 
$(\fc^Z_n)_{n \le 0}$ is not product-type under $\ppf_u$. 
To do this, we check that the random variable $Z_0$ 
does not satisfy the I-cosiness criterion.
Let $Z'=(X',Y')$ and $Z''=(X'',Y'')$ be two copies of the process 
$Z$ under $\ppf_u$, defined on some probability space 
$(\bar{\Omega},\bar{\ac},\bar{\ppf}_u)$, such that both natural 
filtrations $\fc^{Z'}$ and $\fc^{Z''}$ are
immersed in some filtration $\gc$.

For every $n \le 0$, define the copies $A'_n,A''_n$ and $B'_n,B''_n$ 
of the random variables $A_n$ and $B_n$ by the obvious way, and set 
$S_n = \{x \in K_n^4 : A'_nx+B'_n = A''_nx+B''_n\}$.
Then for $U(\ppf)$-almost every $u \in E$, the equalities 
$Y'_n = A'_nX'_n+B'_n+u_{2n}$ and $Y''_n = A''_nX''_n+B''_n+u_{2n}$ 
hold $\bar{\ppf}_u$-almost surely. Therefore, 
$$\one_{[Z'_{2n}=Z''_{2n}]} = \one_{[X'_n = X''_n \in S_n]} 
\le \one_{[X'_n \in S_n]} \ \bar{\ppf}_u \text{-almost surely}.$$ 
But the random variable $S_n$ is $\gc_{2n-1}$-measurable whereas
$X'_n$ is uniform on $K_n^4$ conditionally on $\gc_{2n-1}$ since 
$\fc^{Z'}$ is immersed in $\gc$. Thus
$$\bar{\ppf}_u[Z'_{2n}=Z''_{2n}|\gc_{2n-1}] 
\le \frac{|S_n|}{q_n^4} 
\le \one_{[A'_n=A''_n]} + \frac{1}{q_n} \one_{[A'_n \ne A''_n]}
\ \bar{\ppf}_u \text{-almost surely},$$ 
since $S_n$ contains at most $q_n^3$ points when $A'_n \ne A''_n$. 
Passing to the complements and taking the expectations yields
$$\bar{\ppf}_u[Z'_{2n} \ne Z''_{2n}] 
\ge \Big(1-\frac{1}{q_n}\Big) \bar{\ppf}_u[A'_{n} \ne A''_{n}] 
= \Big(1-\frac{1}{q_n}\Big) \bar{\ppf}_u[Z'_{2n-2} \ne Z''_{2n-2}].$$
By induction, one gets that for every $n \le 0$
$$\bar{\ppf}_u[Z'_0 \ne Z''_0] 
\ge \prod_{k=n+1}^0 \Big(1-\frac{1}{q_k}\Big) 
\times \bar{\ppf}_u[Z'_{2n} \ne Z''_{2n}].$$
If, for some $N > -\infty$, the $\sigma$-fields 
$\fc^{Z'}_N$ and $\fc^{Z''}_N$ are independent, 
then $\bar{\ppf}_u[Z'_{2n} \ne Z''_{2n}] \to 1$ as $n \to -\infty$, so
$$\bar{\ppf}_u[Z'_0 \ne Z''_0] 
\ge \prod_{k \le 0} \Big(1-\frac{1}{q_k}\Big) > 0.$$
The proof is complete. 
\end{proof}

\subsection{Embedding dyadic filtrations in Brownian filtrations}

We start with the two filtrations provided by lemma~\ref{dyadic}. 
By construction, the filtration $(\zc_n)_{n \le 0}$ can be generated by 
some i.i.d. sequence $(\ep_n)_{n \le 0}$ of uniform random variables 
with values in $\{-1,1\}$. 


The filtration $(\uc_n)_{n \le 0}$ is $(a_n)_{n \le 0}$-adic for some 
sequence $(a_n)_{n \le 0}$ taking values $1$ and $2$ only. 
Call $D \subset \zzf_-$ 
the set of all $n \le 0$ such that $a_n=2$. Since $(\uc_n)_{n \le 0}$
is poly-adic and immersed in a product-type filtration, 
$(\uc_n)_{n \le 0}$ is also product-type. 
Therefore, the filtration $(\uc_n)_{n \le 0}$ can be generated by 
some sequence $(\eta_n)_{n \le 0}$ of independent random variables with
$\eta_n$ uniform on $\{-1,1\}$ if $n \in D$, $\eta_n=0$ if $n \notin D$. 

By immersion of $(\uc_n)_{n \le 0}$ in $(\zc_n)_{n \le 0}$, each  
$\eta_n$ is $\zc_n$-measurable and independent of $\zc_{n-1}$. 
So when $n \in D$, $\eta_n$ can be written $\eta_n = H_n \ep_n$, 
where $H_n$ is some $\zc_{n-1}$-random variable taking values 
in $\{-1,1\}$. 

Fix an increasing sequence $(t_n)_{n \le 0}$ of positive real numbers 
such that $t_0=1$ and $t_n \to 0$ as $n \to -\infty$ (e.g. 
$t_n=2^n$ for every $n \le 0$). By symmetry and independence of Brownian 
increments, one may construct a Brownian motion $X$ such that for every 
$n \le 0$, $\ep_n = \sign(X_{t_n}-X_{t_{n-1}})$. Let $Y$ be another  
Brownian motion, independent of $X$. 

Since $\zc_{n-1} \subset \fc^{X,Y}_{t_{n-1}}$ for every $n \le 0$, one gets 
a predictable process $(A_t)_{0 < t \le 1}$ with values in $O_2(\rrf)$ and 
two independent Brownian motions $B$ and $C$ in $\fc^{X,Y}$ 
on the time-interval $[0,1]$ by setting for every $t \in ]t_{n-1},t_n]$,
$$A_t = \left(\begin{array}{cc}
H_n & 0 \\
0 & 1
\end{array}\right) \text{ if } n \in D,\quad 
A_t = \left(\begin{array}{cc}
0 & 1 \\
1 & 0
\end{array}\right) \text{ if } n \notin D,$$ 
and for every $t>0$,
$$\left(\begin{array}{c}
dB_t \\
dC_t
\end{array}\right)
= A_t\left(\begin{array}{c}
dX_t \\
dY_t
\end{array}\right).$$

\begin{theo}
The filtration generated by the Brownian motion $B$ thus defined is 
complementable after $0$, maximal, but not complementable in $\fc^{X,Y}$. 
\end{theo}

\begin{proof}

{\bf Complementability after $0$} 

Let us check that $C$ is a complement after $0$ of $B$, or equivalently that 
$$\forall s \in ]0,1],~\fc^{B,C}_1 \vee \fc^{X,Y}_s = \fc^{X,Y}_1.$$
Since $t_0=1$ and $t_m \to 0+$ as $m \to -\infty$, it is sufficient to consider
instants $s$ which are some subdivision time $t_m$ with $m \le 0$. 
Since for every $n \ge m$, the process $A$ coincides on each 
time-interval $]t_n,t_{n+1}]$ with an $\fc^{X,Y}_{t_n}$-measurable 
random variable, the formula 
$$\left(\begin{array}{c}
dX_t \\
dY_t
\end{array}\right)
= A_t^{-1}\left(\begin{array}{c}
dB_t \\
dC_t
\end{array}\right)$$
enables us to recover $(X,Y)$ from the 
knowledge of $((X_s,Y_s))_{0 \le s \le t_m}$ and $(B,C)$.  

{\bf Maximality}

By proposition~\ref{sufficient condition for maximality : Brownian case},
the maximality of $B$ will follow from its complementability after $0$ 
once we will have proved the equality 
$$\fc^B_1 = \bigcap_{s \in ]0,1]} (\fc^B_1 \vee \fc^{X,Y}_s).$$
The intersection above, over all $s \in ]0,1]$ can be restricted to the 
instants $t_m$ with $m \le 0$.   

It is now convenient to introduce the notations
$$\Delta X_n = (X_t-X_{t_{n-1}})_{t_{n-1} \le t \le t_n} \text{ and }
\ep_n\Delta X_n = (\ep_n(X_t-X_{t_{n-1}}))_{t_{n-1} \le t \le t_n}.$$
Recall that $\ep_n = \sign(X_{t_n}-X_{t_{n-1}})$. Therefore,
$\sigma(\Delta X_n) = \sigma(\ep_n) \vee \sigma(\ep_n\Delta X_n)$, 
with $\sigma(\ep_n)$ and $\sigma(\ep_n\Delta X_n)$ independent
by symmetry of Brownian increments.

We define in the same way the random variables 
$\Delta Y_n$, $\Delta B_n$, $\Delta C_n$ and $\eta_n\Delta B_n$. Then
$$\Delta B_n = H_n\Delta X_n = \eta_n\ep_n\Delta X_n 
\text{ and } \Delta C_n = \Delta Y_n \text{ if } n \in D,$$ 
$$\Delta B_n = \Delta Y_n 
\text{ and } \Delta C_n = \Delta X_n \text{ if } n \in D^c.$$ 
Moreover, when $n \in D$, 
$\eta_n = \sign(B_{t_n}-B_{t_{n-1}})$ is independent of
$\eta_n\Delta B_n = \ep_n\Delta X_n$.

Therefore, $\fc^B_1 = \ac \vee \bc$, with 
$$\ac = \sigma((\eta_n)_{n \in D}),\quad 
\bc = \sigma((\ep_n\Delta X_n)_{n \in D}) \vee \sigma((\Delta Y_n)_{n \in D^c}).$$
For $n \in \zzf_-$, set $D_n = D \cap ]-\infty,n]$,
$D_n^c = ]-\infty,n] \setminus D$, and
$$\cc_n = \fc^\ep_n, \quad \dc_n = \sigma((\ep_k\Delta X_k)_{k \in D_n^c}) 
\vee \sigma((\Delta Y_k)_{k \in D_n}).$$
Then $\fc^\ep_n \vee \fc^C_{t_n} = \cc_n \vee \dc_n$.

The maximality of $\fc^\eta$ in $\fc^\ep$ yields the equality 
$$\ac = \bigcap_{n \le 0} (\ac \vee \cc_n) \mod \ppf.$$
By independence of $B$ and $C$, the $\sigma$-fields $\bc$ and $\dc_0$ are 
independent, so corollary~\ref{independence allows exchange} 
applies and the following exchange property holds
$$\bc = \bc \vee \dc_{-\infty} = \bigcap_{n \le 0} (\bc \vee \dc_n) \mod \ppf.$$
Since the three sequences $(\ep_n)_{n \le 0}$, $(\ep_n \Delta X_n)_{n \le 0}$
and $(\Delta Y_n)_{n \le 0}$ are independent, the $\sigma$-fields
$\ac \vee \cc_0 = \fc^{\ep}_0$ and $\bc \vee \dc_0 = \fc^{\ep \Delta X,Y}_0$ 
are independent. Hence, lemma~\ref{independent exchange} yields  
$$\fc^B_1 = \ac \vee \bc 
= \bigcap_{n \le 0} (\ac \vee \bc \vee \cc_n \vee \dc_n)
= \bigcap_{n \le 0} (\fc^B_1 \vee \fc^\ep_n \vee \fc^C_{t_n})
= \bigcap_{n \le 0} (\fc^B_1 \vee \fc^{X,Y}_{t_n}) \mod \ppf.$$
This proves the maximality of $B$. 

{\bf Non-complementability}

Keep the notations introduced in the proof of the maximality and set  
$\ep:=(\ep_n)_{n \le 0}$, $\eta := (\eta_n)_{n \le 0}$. 
Remind that $\ep$, $(\ep_n \Delta X_n)_{n \le 0}$ and $(\Delta Y_n)_{n \le 0}$ 
are independent families of independent random variables and that 
$\fc^B_1$ is the $\sigma$-field 
generated by $\eta$, $(\ep_n \Delta X_n)_{n \in D}$ and $(\Delta Y_n)_{n \in D^c}$. 

The filtration $(\fc^{X,Y}_{t_n})_{n \le 0}$ can be splitted into 
three independent parts, namely 
$$\fc^{X,Y}_{t_n} = \fc^\ep_n 
\vee \sigma((\ep_k \Delta X_k)_{k \in D_n} \cup (\Delta Y_k)_{k \in D_n^c})
\vee \sigma((\Delta Y_k)_{k \in D_n} \cup (\ep_k \Delta X_k)_{k \in D_n^c}).$$
The second part is a function of $B$ whereas the third part is independent 
of $(\ep,B)$. By independent enlargement, we get that for 
$B(\ppf)$-almost every $b \in C([0,1],\rrf)$,
the filtration $(\fc^\ep_n)_{n \le 0}$ is immersed in 
$(\fc^{X,Y}_{t_n})_{n \le 0}$ under $\ppf[\cdot|B=b]$. 

But $\eta$ is some measurable function $\Phi$ of $B$ and is 
also a function of $\ep$. 
Since $\ep$, $(\ep_n \Delta X_n)_{n \in D}$ and $(\Delta Y_n)_{n \in D^c}$ 
are independent, the law of $\ep$ under  
$\ppf[\cdot|B=b]$ coincides with the law of $\ep$ under 
$\ppf[\cdot|\eta = \Phi(b)]$. 

Since $\eta$ generates the same 
$\sigma$-field as the random variable $U$ of lemma~\ref{dyadic}, 
we derive that for $B(\ppf)$-almost every $b \in C([0,1],\rrf)$,
the filtration $(\fc^\ep_n)_{n \le 0}$ is $(2/a_n)$-adic but not 
product-type under $\ppf[\cdot|B=b]$. 
But this filtration is immersed in 
$(\fc^{X,Y}_{t_n})_{n \le 0}$ under $\ppf[\cdot|B=b]$, 
hence by Vershik's theorem (theorem~\ref{Vershik} in the present paper),
$(\fc^{X,Y}_{t_n})_{n \le 0}$ cannot be product-type 
so $\fc^{X,Y}$ cannot be Brownian under $\ppf[\cdot|B=b]$. 
Thus, the Brownian filtration $\fc^B$ is not complementable in $\fc^{X,Y}$.  
\end{proof}

\begin{lemm}\label{independent exchange} 
Let $\ac$, $\bc$ be two sub-$\sigma$-fields and $(\cc_n)_{n \le 0}$, 
$(\dc_n)_{n \le 0}$ be two filtrations of the probability space 
$(\Omega,\tc,\ppf)$. If 
$$\ac = \bigcap_{n \le 0} (\ac \vee \cc_n) \mod \ppf, \quad
\bc = \bigcap_{n \le 0} (\bc \vee \dc_n) \mod \ppf,$$
and if $\ac \vee \cc_0$ and $\bc \vee \dc_0$ are independent, then
$$\ac \vee \bc 
= \bigcap_{n \le 0} (\ac \vee \bc \vee \cc_n \vee \dc_n) \mod \ppf.$$
\end{lemm}

\begin{proof}
Since both sides of the equality to be proved are sub-$\sigma$-fields 
of $\ac \vee \bc \vee \cc_0 \vee \dc_0$, it is sufficient to prove 
that for every $Z \in L^1(\ac \vee \cc_0 \vee \bc \vee \dc_0)$, 
one has 
$$\eef[Z|\ac \vee \bc] 
= \eef\Big[ Z \Big| 
\bigcap_{n \le 0} (\ac \vee \bc \vee \cc_n \vee \dc_n)\Big].$$
Considering only random variables $Z = XY$ with $X \in L^1(\ac \vee \cc_0)$ and 
$Y \in L^1(\bc \vee \dc_0)$ is sufficient since these random variables 
span a dense subspace in $L^1(\ac \vee \cc_0 \vee \bc \vee \dc_0)$. 
Given $Z$ as above, one has $\eef[Z|\ac \vee \cc_0 \vee \bc] = X \eef[Y|\bc]$, 
so
$$\eef[Z|\ac \vee \bc] = \eef[X|\ac \vee \bc] \eef[Y|\bc] 
= \eef[X|\ac] \eef[Y|\bc],$$
since $\sigma(X) \vee \ac$ is independent of $\bc$. 
In the same way, one gets that for every $n \le 0$,  
$$\eef[Z|\ac \vee \bc \vee \cc_n \vee \dc_n] 
= \eef[X|\ac \vee \cc_n] \eef[Y|\bc \vee \dc_n].$$
Thus, taking the limit as $n \to -\infty$ yields the result by the 
martingale convergence theorem and the assumption. 
\end{proof}

The particular case where $\bc$ and the $\sigma$-fields $\cc_n$ are 
equal to $\{\emptyset,\Omega\}$ yields the following classical and 
useful result. 

\begin{coro}\label{independence allows exchange}
Let $\ac$ be a sub-$\sigma$-field and $(\dc_n)_{n \le 0}$ be a filtration 
of the probability space $(\Omega,\tc,\ppf)$. If $\ac$ and $\dc_0$ are 
independent, then 
$$\ac = \bigcap_{n \le 0} (\ac \vee \dc_n) \mod \ppf.$$
\end{coro}

\section{A complementable factor arising from a 
non-complemen\-table filtration}
\label{complementable filtration yielding a complementable factor}

\subsection{Definition of a uniform randomised decimation 
process}

We denote by $\{a,b\}^\infty$ the set of all infinite words on the 
alphabet $\{a,b\}$, namely the set of all maps from 
$\nnf = \{1,2,\ldots\}$ to $\{a,b\}$. We endow this 
set with the uniform probability measure $\mu$: 
a random infinite word $X$ is chosen according to $\mu$ 
if the successive letters 
$X(1),X(2),\ldots$ form a sequence of independent and uniform 
random variables taking values in $\{a,b\}$.  

We denote by $\pc(\nnf)$ the power set of $\nnf$, i.e. the set 
of all subsets of $\nnf$. Given $p \in ]0,1[$, we define the 
probability measure $\nu_p$ on $\pc(\nnf)$ as follows: 
the law of a random subset $I$ of $\nnf$ is $\nu_p$ if  
$\one_I(1),\one_I(2),\ldots$ form an i.i.d. sequence of Bernoulli
random variables with parameter $p$. Equivalently, this means 
that $\ppf[F \subset I] = p^{|F|}$ for every finite subset 
$F$ of $\nnf$. In this case, we note that almost surely, 
$I$ is infinite with infinite complement. 
The law $\nu := \nu_{1/2}$ will be called the uniform law on 
$\pc(\nnf)$. 

When $A$ is an infinite subset of $\nnf$, we denote by 
$\psi_A(1)<\psi_A(2)<\ldots$ its elements. This defines an 
increasing map $\psi_A$ from $\nnf$ to $\nnf$ whose range is 
$A$. Conversely, for every increasing map $f$ from 
$\nnf$ to $\nnf$, there is a unique infinite subset $A$ of $\nnf$, 
namely the range of $f$, such that $f = \psi_A$. These remarks
lead to the following statement. 

\begin{lemm}\label{composition of random extractions}
Let $I$ and $J$ be independent random infinite subsets of $\nnf$ 
with respective laws $\nu_p$ and $\nu_q$, and 
$R = \psi_I \circ \psi_J(\nnf) = \psi_I(J)$ be the range of 
$\psi_I \circ \psi_J$. Then $\psi_I \circ \psi_J = \psi_R$ and 
the law of $R$ is $\nu_{pq}$. 
\end{lemm}

\begin{proof}
The equality $\psi_I \circ \psi_J = \psi_R$ follows from the 
remarks above. Let $F$ be a finite subset of $\nnf$. By 
injectivity of $\psi_I$, 
$$[F \subset R] = [F \subset I\ ;\ \psi_I^{-1}(F) \subset J]$$
and $[F \subset I] = [|\psi_I^{-1}(F)|=|F|]$, therefore by 
independence of $I$ and $J$, 
$$\ppf[F \subset R\ |\ \sigma(I)] 
= \one_{[F \subset I]}\ \ppf[\psi_I^{-1}(F) \subset J\ |\ \sigma(I)] 
= \one_{[F \subset I]}\ q^{|\psi_I^{-1}(F)|}
= \one_{[F \subset I]}\ q^{|F|}.$$
Thus $\ppf[F \subset R] = \ppf[F \subset I] q^{|F|} = (pq)^{|F|}$.  
\end{proof}

Here is another property that we will use to define the uniform 
randomised decimation process on $\{a,b\}$, and also later, in 
the proof of proposition~\ref{rejected subsequences}.

\begin{lemm}\label{complementary subwords}
Let $X$ be a uniform random word on $\{a,b\}^\infty$. 
Let $I$ be a random subset of $\nnf$ with law $\nu_p$, 
independent of $X$. Then 
\begin{itemize}
\item $I$, $X \circ \psi_I$, $X \circ \psi_{I^c}$ are independent 
\item $X \circ \psi_I$, $X \circ \psi_{I^c}$ are uniform random words 
on $\{a,b\}^\infty$.
\end{itemize}
\end{lemm}

\begin{proof}
Almost surely, $I$ is infinite with infinite complement, so the random 
maps $\psi_I$ and $\psi_{I^c}$ are well-defined.  
The integers $\psi_I(1),\psi_{I^c}(1),\psi_I(2),\psi_{I^c}(2)\ldots$ 
are distinct, so conditionally on $I$, the random variables 
$X(\psi_I(1)), X(\psi_{I^c}(1)), X(\psi_I(2)), X(\psi_{I^c}(2)),...$ 
are independent and uniform on $\{a,b\}$. The result follows.   
\end{proof}


\begin{defi}
Call $\pc'(\nnf)$ the set of all infinite subsets of $\nnf$. 
A uniform randomised decimation process in the alphabet $\{a,b\}$ 
is a stationary Markow chain $(X_n,I_n)_{n \in \zzf}$ with values in 
$\{a,b\}^\infty \times \pc'(\nnf)$ defined as follows:
for every $n \in \zzf$,  
\begin{enumerate}
\item the law of $(X_n,I_n)$ is $\mu \otimes \nu$;
\item $I_n$ is independent of $(X_{n-1},I_{n-1})$ and uniform on $\pc(\nnf)$; 
\item $X_n = X_{n-1} \circ \psi_{I_n}$. 
\end{enumerate}
\end{defi}

Such a process is well-defined and unique in law since the law 
$\mu \otimes \nu$ is invariant by the transition kernel given by 
conditions 2 and 3 above, thanks to lemma~\ref{complementary subwords}. 
Moreover, $(I_n)_{n \in \zzf}$ is a sequence of 
innovations for the filtration $\fc^{X,I}$. Therefore, the 
filtration $\fc^{X,I}$ has independent increments or is 
locally of product-type, according to Laurent's 
terminology~\cite{Laurent (cosiness standardness)}. 

This process is a kind of randomisation of Vershik's decimation process 
given in example~3 of~\cite{Vershik}. Indeed, Vershik's decimation 
process is equivalent to the process that we would get by choosing 
the random sets $I_n$ uniformly among the set of all even
positive integers and the set of all odd positive integers. 
Although Vershik's decimation process generates a non-standard filtration, 
we will show that our randomised process generates a standard one. 

\begin{theo}\label{product-type}
The uniform randomised decimation process on the alphabet 
$\{a,b\}$ generates a product-type filtration. 
\end{theo}

\subsection{Proof of theorem~\ref{product-type}}

We have seen that the filtration $\fc^{X,I}$ admits $(I_n)_{n \in \zzf}$ 
as a sequence of innovations. Each innovation has diffuse law. 
Therefore, to prove that the filtration 
$(\fc^{X,I}_n)_{n \le 0}$, or equivalently, the filtration 
$(\fc^{X,I}_n)_{n \in \zzf}$ is product-type, 
it suffices to check Vershik's first level criterion 
(see reminders further and definition~2.6 and theorem~2.25 
in~\cite{Laurent (cosiness standardness)}). Concretely, 
we have to check any random variable in $L^1(\fc^{X,I}_0,\rrf)$ 
can be approached in $L^1(\fc^{X,I}_0,\rrf)$ by measurable 
functions of finitely many innovations of $(\fc^{X,I})_{n \le 0}$. 

The innovations $(I_n)_{n \in \zzf}$ are inadequate to do this, 
since the random variable $X_0$ is independent of the whole 
sequence $(I_n)_{n \in \zzf}$, so functions of the $(I_n)_{n \in \zzf}$
cannot approach non-trivial functions of $X_0$. 
Therefore, we will have to construct new innovations. The 
next lemma gives us a general procedure to get some. 

\begin{lemm}\label{new innovation}
Fix $n \in \zzf$. Let $\Phi$ be some $\fc^{X,I}_{n-1}$-measurable map 
from $\nnf$ to $\nnf$. If $\Phi$ is almost surely bijective, then 
the random variable $J_n = \Phi(I_n)$ is independent of $\fc^{X,I}_{n-1}$
and uniform on $\pc(\nnf)$. 
\end{lemm}

\begin{proof}
For every finite subset $F$ of $\nnf$,  
$$P[F \subset J_n|\fc^{X,I}_{n-1}] 
= P[\Phi^{-1}(F) \subset I_n|\fc^{X,I}_{n-1}] 
= (1/2)^{|\Phi^{-1}(F)|} = (1/2)^{|F|} \text{ almost surely.}$$ 
The result follows.
\end{proof}

Actually, the proof of theorem~\ref{product-type}   is similar 
to the proof of the standardness of the erased-words filtration 
by S.~Laurent~\cite{Laurent (erased)} and uses the same tools, 
namely canonical coupling and cascaded permutations. 

\begin{defi} {\bf (Canonical word and canonical coupling)}

The infinite canonical word $C$ on the alphabet 
$\{a,b\}$ is the word $abab \cdots$, namely the map from $\nnf$ 
to $\{a,b\}$ which sends the odd integers on $a$ and the even 
integers on $b$. 

If $x$ is an infinite word $x$ on the alphabet $\{a,b\}$, 
namely a map from $\nnf$ to $\{a,b\}$, we set for every $i \in \nnf$, 
\begin{eqnarray*}
\phi_x(i) = 2q-1 
\text{ if $x(i)$ is the $q$-th occurence of the letter $a$ in $x$,} \\
\phi_x(i) = 2q 
\text{ if $x(i)$ is the $q$-th occurence of the letter $b$ in $x$.} 
\end{eqnarray*}
\end{defi}

\begin{lemm}\label{canonical coupling}
By definition, the map $\phi_x$ thus defined from $\nnf$ to $\nnf$ 
is injective and satisfies the equality $x = C \circ \phi_x$. 
When each possible letter $a$ or $b$ appears 
infinitely many times in $x$, $\phi_x$ is a permutation of $\nnf$, 
(called {\it canonical coupling} by S.~Laurent). 
\end{lemm}

Roughly speaking, if $x$ is a typical word of $\{a,b\}^\infty$ endowed 
with the uniform law, the asymptotic proportions of $a$ and $b$ 
are $1/2$ are $1/2$, so $\phi_x$ is asymptotically close to the 
identity map. 


\begin{defi}\label{new innovations}
{\bf (New innovations and cascaded permutations)}

Let $\Omega'$ be the almost sure event on which
\begin{itemize} 
\item each possible letter $a$ or $b$ appears infinitely many times in 
the infinite word $X_0$;  
\item each subset $I_n$ is infinite.
\end{itemize}
On $\Omega'$, we define by recursion a sequence $(\Phi_n)_{n \ge 0}$ of 
random permutations of $\nnf$ and a sequence $(J_n)_{n \ge 1}$ of random 
infinite subsets of $\nnf$ by setting $\Phi_0 = \phi_{X_0}$ and, 
for every $n \ge 1$, 
\begin{equation}\label{recursion formula}
J_n = \Phi_{n-1}(I_n) \text{ and } 
\Phi_{n-1} \circ \psi_{I_n} = \psi_{J_n} \circ \Phi_n. 
\end{equation}
\end{defi}

Let us check that the inductive construction above actually works $\Omega'$. 

On $\Omega'$, the map $\Phi_0 = \phi_{X_0}$ is bijective by 
lemma~\ref{canonical coupling}. 

Once we know that $\Phi_{n-1}$ is a random permutations 
of $\nnf$, the map $\Phi_{n-1} \circ \psi_{I_n}$ is a random 
injective map from $\nnf$ to $\nnf$ with range $\Phi_{n-1}(I_n) = J_n$.  
Therefore, $J_n$ is infinite and the map $\Phi_n$ is well defined by 
equation~\ref{recursion formula}:
for every $k \in \nnf$, $\Phi_n(k)$ is the rank of the integer 
$\Phi_{n-1}(\psi_{I_n}(k))$ in the set $J_n$. Moreover, $\Phi_n$ 
is a permutation of $\nnf$. 

Informally, the {\it cascaded permutations} $(\Phi_n)_{n \ge 0}$ are induced by 
$\Phi_0 = \phi_{X_0}$ and the successive extractions. 
More precisely, equation~\ref{recursion formula} is represented by a 
commutative diagramm which gives the correspondance between 
the positions of a same letter in different words. 
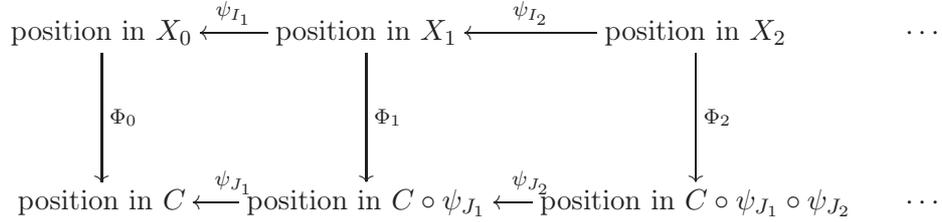
\begin{displaymath}
\xymatrixcolsep{5mm}
\xymatrix@R+2pc{
\text{position in }X_0 \ar[d]^{\Phi_0} 
& \text{position in } X_1 \ar[l]_{\psi_{I_1}} \ar[d]^{\Phi_1} 
& \text{position in } X_2 \ar[l]_{\psi_{I_2}} \ar[d]^{\Phi_2} 
& \cdots \\
\text{position in } C 
& \text{position in } C \circ \psi_{J_1} \ar[l]_{\psi_{J_1}} 
& \text{position in } C \circ \psi_{J_1} \circ \psi_{J_2} \ar[l]_{\psi_{J_2}} 
& \cdots \\  
}
\end{displaymath}

Here is a realisation of the first three steps. The boldface numbers 
form the subsets $I_1,J_1,I_2,J_2,...$. Among the arrows representing 
$\phi_{X_0}$, the plain arrows 
(from elements in $I_1$ to elements in $J_1$) provide the permutation
$\phi_{X_0,I_1}$ by renumbering of the elements. 

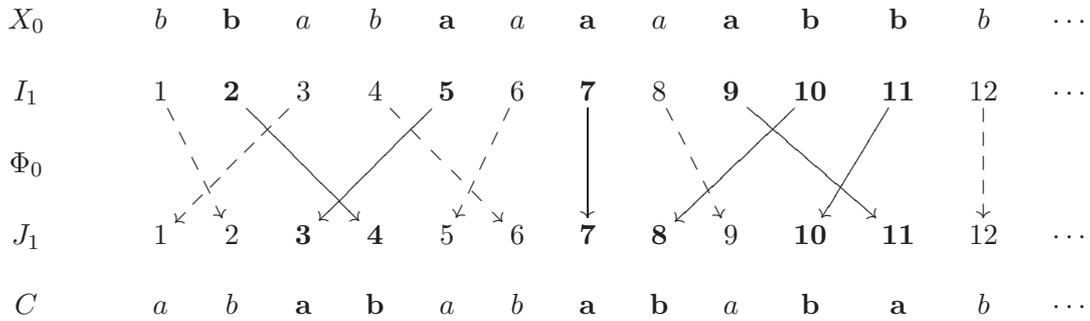
\begin{displaymath}
\xymatrixcolsep{5mm}
\xymatrix@R-1pc{
X_0& & b & {\bf b} & a & b & 
{\bf a} & a & {\bf a} & a & 
{\bf a} & {\bf b} & {\bf b} & b & \cdots \\
I_1 && 1\dar[ddr] & {\bf 2}\ar[ddrr] & 3\dar[ddll] & 4\dar[ddrr] & 
{\bf 5}\ar[ddll] & 6\dar[ddl] & {\bf 7}\ar[dd] & 8\dar[ddr] &
{\bf 9}\ar[ddrr] & {\bf 10}\ar[ddll] & {\bf 11}\ar[ddl] & 12\dar[dd] & \cdots\\ 
\Phi_0 & & & & & & & & & & & & \\
J_1 & & 1 & 2 & {\bf 3} & {\bf 4} & 
5 & 6 & {\bf 7} & {\bf 8} & 9 & 
{\bf 10} & {\bf 11} & 12 & \cdots \\
C & & a & b & {\bf a} & {\bf b} & a & b & {\bf a} 
& {\bf b} & a & {\bf b} & {\bf a} & b & \cdots       
}
\end{displaymath}
\vskip 5mm
\begin{displaymath}
\xymatrixcolsep{5mm}
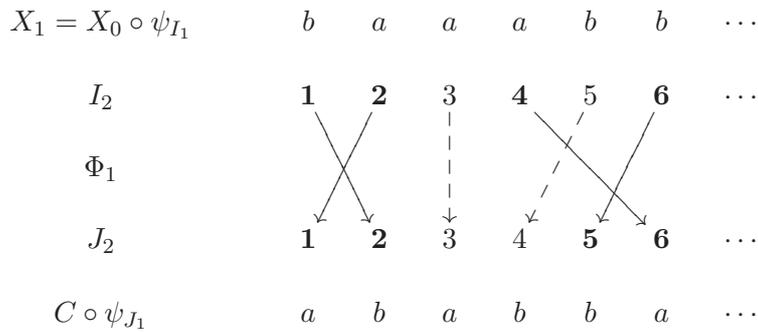
\xymatrix@R-1pc{
X_1 = X_0 \circ \psi_{I_1}& & b & a & a & a & b & b & \cdots \\
I_2 && {\bf 1}\ar[ddr] & {\bf 2}\ar[ddl] & 3\dar[dd] 
& {\bf 4}\ar[ddrr] & 5\dar[ddl] & {\bf 6}\ar[ddl] & \cdots\\ 
\Phi_1 & & & & & &  \\
J_2 & & {\bf 1} & {\bf 2} & 3 & 4 & {\bf 5} & {\bf 6} & \cdots \\
C \circ \psi_{J_1} & & a & b & a & b & b & a & \cdots       
}
\end{displaymath}
\vskip 5mm
\begin{displaymath}
\xymatrixcolsep{5mm}
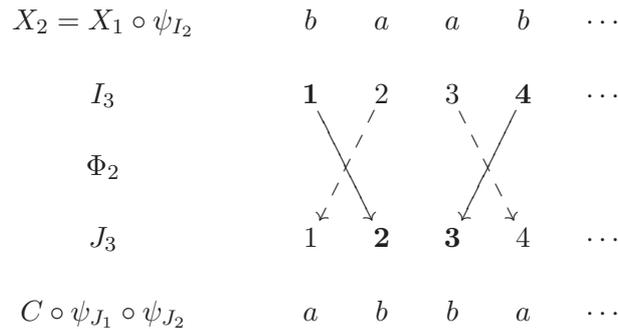
\xymatrix@R-1pc{
X_2 = X_1 \circ \psi_{I_2} & & b & a & a & b & \cdots \\
I_3 && {\bf 1}\ar[ddr] & 2\dar[ddl] & 3\dar[ddr] 
& {\bf 4}\myar[ddl] & \cdots\\ 
\Phi_2 & & & & & &  \\
J_3 & & 1 & {\bf 2} & {\bf 3} & 4 & \cdots \\
C \circ \psi_{J_1} \circ \psi_{J_2} & & a & b & b & a & \cdots       
}
\end{displaymath}

\vfill\eject

\begin{lemm}\label{relations involving new innovations}
On the almost sure event $\Omega'$, the following properties hold 
for every $n \ge 1$, 
\begin{enumerate}
\item $J_n$ is independent of $\fc^{X,I}_{n-1}$ and is uniform on $\pc(\nnf)$. 
\label{independence}
\item $\sigma(X_0,J_1,\ldots,J_n) = \sigma(X_0,I_1,\ldots,I_n)$.
\label{sigma-field} 
\item the random map $\Phi_n$ is 
$\sigma(X_0,J_1,\ldots,J_n)$-measurable. \label{measurability}
\item $\phi_{X_0} \circ \psi_{I_1} \circ \cdots \circ \psi_{I_n} 
= \psi_{J_1} \circ \cdots \circ \psi_{J_n} \circ \Phi_n$.
\label{cascaded permutations}
\item $X_n = C \circ \psi_{J_1} \circ \cdots \circ \psi_{J_n}
\circ \Phi_n$.
\label{formula for $X_n$}
\end{enumerate}
\end{lemm}

\begin{proof}
Since $\Phi_0=\phi_{X_0}$, properties~\ref{sigma-field}, 
\ref{measurability}, \ref{cascaded permutations}, \ref{formula for $X_n$} 
above hold with $n$ replaced by $0$.  

Let $n \ge 1$. Assume that  
properties~\ref{sigma-field}, \ref{measurability}, 
\ref{cascaded permutations}, \ref{formula for $X_n$} 
hold with $n$ replaced by $n-1$. 

Then by lemma~\ref{new innovation}, property~\ref{independence} holds. 

By definition and by the induction hypothesis, the random set
$J_n = \Phi_{n-1}(I_n)$ is $\sigma(X_0,I_1,\ldots,I_n)$-measurable.
Conversely, since $I_n = \Phi_{n-1}^{-1}(J_n)$, the knowledge of 
$\Phi_{n-1}$ and $J_n$ is sufficient to recover $I_n$, 
so property~\ref{sigma-field} holds.  

For every $k \in \nnf$, $\Phi_n(k)$ is the rank of the integer 
$\Phi_{n-1}(\psi_{I_n}(k))$ in the set $\Phi_{n-1}(I_n)$. 
Thus the random map $\Phi_n$ is a measurable for 
$\sigma(X_0,J_1,\ldots,J_{n-1},I_n) = \sigma(X_0,J_1,\ldots,J_n)$. 
Therefore, property~\ref{measurability} holds. 

By induction hypothesis and by formula~\ref{recursion formula},
\begin{eqnarray*}
\phi_{X_0} \circ \psi_{I_1} \circ \cdots \circ \psi_{I_n}  
&=& (\phi_{X_0} \circ \psi_{I_1} \circ \cdots \circ \psi_{I_{n-1}}) 
\circ \psi_{I_n} \\
&=& (\psi_{J_1} \circ \cdots \circ \psi_{J_{n-1}} \circ \Phi_{n-1}) 
\circ \psi_{I_n} \\
&=& (\psi_{J_1} \circ \cdots \circ \psi_{J_{n-1}}) 
\circ (\Phi_{n-1} \circ \psi_{I_n}) \\
&=& (\psi_{J_1} \circ \cdots \circ \psi_{J_{n-1}}) 
\circ (\psi_{J_n} \circ \Phi_n) \\  
&=&  \psi_{J_1} \circ \cdots \circ \psi_{J_n} \circ \Phi_n,
\end{eqnarray*}
so
\begin{eqnarray*}
X_n &=& X_0 \circ \psi_{I_1} \circ \cdots \circ \psi_{I_n} \\
&=& C \circ \phi_{X_0} \circ \psi_{I_1} \circ \cdots \circ \psi_{I_n} \\  
&=& C \circ \psi_{J_1} \circ \cdots \circ \psi_{J_n} \circ \Phi_n,
\end{eqnarray*}
which yields properties~\ref{cascaded permutations} 
and~\ref{formula for $X_n$}. 

Lemma~\ref{new innovations} follows by recursion. 
\end{proof}

The next result shows that the innovations $(J_n)_{n \ge 1}$ 
constructed above provide better 
and better approximations of $X_n$ as $n \to +\infty$. 

\begin{lemm}\label{better and better approximations}
Fix $L \in \nnf$. Then 
$\ppf\big[X_n = C \circ \psi_{J_1} \circ \cdots \circ \psi_{J_n}
\text{ on } \odc 1,L \fdc \big] \to 1$ as $n \to +\infty$. 
\end{lemm}

\begin{proof}
By equality~\ref{formula for $X_n$}, it suffices to check that, 
$\ppf(E_n) \to 1$ as $n \to +\infty$, where $E_n$ is the event
``$\Phi_n$ coincides on $\odc 1,L \fdc$ with the identity map''.

By lemma~\ref{composition of random extractions}, 
$\psi_{I_1} \circ \cdots \circ \psi_{I_n} = \psi_{A_n}$ and 
$\psi_{J_1} \circ \cdots \circ \psi_{J_n} = \psi_{B_n}$, 
where $A_n$ and $B_n$ are random subsets of $\nnf$ with 
law $\nu_{p_n}$, where $p_n= 2^{-n}$.

Therefore, by property~\ref{cascaded permutations} 
of lemma~\ref{relations involving new innovations},
$\phi_{X_0} \circ \psi_{A_n} = \psi_{B_n} \circ \Phi_n$, 
so for each $k \in \nnf$, $\Phi_n(k)$ is the rank of the integer 
$\phi_{X_0}(\psi_{A_n}(k))$ in the set $\phi_{X_0}(A_n) = B_n$.

Thus, the event $E_n$ holds if and only if the $L$ first elements 
of the set $\phi_{X_0}(A_n)$ in increasing order are 
$\phi_{X_0}(\psi_{A_n}(1)),\ldots,\phi_{X_0}(\psi_{A_n}(L))$.

Set $\tau_{n,k} = \psi_{A_n}(k)$ for every $k \in \nnf$. 
Since the law of $A_n$ is $\nu_{p_n}$, the random variables 
$\tau_{n,1}, \tau_{n,2}-\tau_{n,1}, \tau_{n,3}-\tau_{n,2},...$
are independent and geometric with parameter $p_n$.  

We have noted that 
$$E_n = \big[ \forall k \ge L+1, \phi_{X_0}(\tau_{n,1}) < \ldots 
< \phi_{X_0}(\tau_{n,L}) < \phi_{X_0}(\tau_{n,k}) \big]$$ 
Roughly speaking, the probability of this event tends to $1$ because 
$\phi_{X_0}$ is close to the identity map and the set 
$A_n$ gets sparser and sparser as $n$ increases to infinity.
Let us formalize this argument. 

Since $X_0$ is uniform on $\{a,b\}^\infty$, 
the random variables $(\eta_i)_{i \ge 1} = (\one_{[X_0(i)=b]})_{i \ge 1}$ form 
an i.i.d. sequence of Bernoulli random variables with parameter $1/2$. 
For every $t \in \nnf$, the random variable $S_t = \eta_1 + \cdots + \eta_t$ 
counts the number of $b$ in the subword $X_0(\odc 1,t \fdc)$, whereas $t-S_t$ 
counts the number of $a$ in the subword $X_0(\odc 1,t \fdc)$, so 
by definition of $\phi_{X_0}$, 
\begin{eqnarray*}
\phi_{X_0}(t) = 2(t-S_t)-1 & &\text{ if } X_0(t)=a, \\
\phi_{X_0}(t) = 2S_t & & \text{ if } X_0(t)=b.
\end{eqnarray*}

Given $t_1<t_2$ in $\nnf$, the inequality 
$\max(S_{t_1},t_1-S_{t_1})<\min(S_{t_2},t_2-S_{t_2})$ 
implies $\phi_{X_0}(t_1)<\phi_{X_0}(t)$ for every integer $t \ge t_2$. 
Therefore, 
$$E_n \supset \big[ \forall k \in \odc 1,L \fdc, 
\max(S_{\tau_{n,k}},\tau_{n,k}-S_{\tau_{n,k}})
< \min(S_{\tau_{n,k+1}},\tau_{n,k+1}-S_{\tau_{n,k+1}}) \big].$$ 
Thus it suffices to prove that for any fixed $k \in \nnf$, 
$$p_{n,k} := \ppf \big[ 
\max(S_{\tau_{n,k}},\tau_{n,k}-S_{\tau_{n,k}})
\ge \min(S_{\tau_{n,k+1}},\tau_{n,k+1}-S_{\tau_{n,k+1}}) \big] 
\to 0 \text{ as } n \to +\infty.$$ 

Since $X_0$ is independent of $I_1,\ldots,I_n$, 
the sequence $(S_t)_{t \ge 1}$ is independent 
of the sequence $(\tau_{n,k})_{k \ge 1}$. Moreover,  
$(S_t)_{t \ge 1}$ has the same law as $(t-S_t)_{t \ge 1}$
and $S_{\tau_{n,k+1}} -S_{\tau_{n,k}}$ has the same law as $S_{\tau_{n,1}}$. 
Therefore, for every integer $x \ge 1$, 
\begin{eqnarray*}
p_{n,k} 
&\le& 2 \ppf \big[ S_{\tau_{n,k+1}} \le \max(S_{\tau_{n,k}},\tau_{n,k}-S_{\tau_{n,k}}) \big] \\
&=& 2\ppf \big[ S_{\tau_{n,k+1}}-S_{\tau_{n,k}} \le \max(0,\tau_{n,k}-2S_{\tau_{n,k}}) \big] \\
&\le& 2\ppf \big[ S_{\tau_{n,k+1}} -S_{\tau_{n,k}} \le x-1 \big] 
+ 2\ppf \big[ \tau_{n,k}-2S_{\tau_{n,k}} \ge x\big] \\
&=& 2\ppf \big[ S_{\tau_{n,1}} \le x-1 \big] 
+ \ppf \big[ |2S_{\tau_{n,k}}-\tau_{n,k}| \ge x \big].
\end{eqnarray*}

On the one hand, the random variable $S_{\tau_{n,1}}$ is binomial with parameters 
$\tau_{n,1}$ and $1/2$ conditionally on $\tau_{n,1}$, so its generating 
function is given by 
\begin{eqnarray*}
\eef[z^{S_{\tau_{n,1}}}] 
&=& \eef\big[\eef[z^{S_{\tau_{n,1}}}\ | \sigma(\tau_{n,1})]\big] \\
&=& \eef\Big[\Big(\frac{1+z}{2}\Big)^{\tau_{n,1}}\Big] \\
&=& \frac{p_n(1+z)/2}{1-(1-p_n)(1+z)/2} \\
&=& \frac{p_n(1+z)}{1+p_n-(1-p_n)z}\\
&=& \frac{p_n(1+z)}{1+p_n} \sum_{m=0}^{+\infty} 
\Big(\frac{1-p_n}{1+p_n} \Big)^m z^m
\end{eqnarray*}
This yields the law of $S_{\tau_{n,1}}$, namely 
$$P[S_{\tau_{n,1}} = 0] = \frac{p_n}{1+p_n},$$
$$P[S_{\tau_{n,1}} = s]
= \frac{p_n}{1+p_n} 
\Big(\Big(\frac{1-p_n}{1+p_n} \Big)^{s-1}+\Big(\frac{1-p_n}{1+p_n} \Big)^s\Big)
= 2p_n \frac{{(1-p_n)}^{s-1}}{{(1+p_n)}^{s+1}}
\text{ if } s \ge 1.$$
Therefore, $\ppf[S_{\tau_{n,1}} = s] \le 2p_n$ for every $s \ge 0$, so
$\ppf \big[ S_{\tau_{n,1}} \le x-1 \big] \le 2p_nx$.

On the other hand, $(2S_t-t)_{t \ge 0}$ is a simple symmetric random 
walk on $\zzf$, independent of $\tau_{n,k}$ so
$$\eef[2S_{\tau_{n,k}}-\tau_{n,k}|\sigma(\tau_{n,k})] = 0 \text{ and }
\var(2S_{\tau_{n,k}}-\tau_{n,k}|\sigma(\tau_{n,k})) = \tau_{n,k}.$$
Therefore, 
$$\eef[2S_{\tau_{n,k}}-\tau_{n,k}] = 0 \text{ and }
\var(2S_{\tau_{n,k}}-\tau_{n,k}) = \var(0) + \eef[\tau_{n,k}] = k/p_n,$$
so Bienaym\'e-Tchebicheff's inequality yields 
$\ppf \big[ |2S_{\tau_{n,k}}-\tau_{n,k}| \ge x \big] \le (k/p_n)x^{-2}$. 

Hence, for every $n$ and $x$ in $\nnf$, $p_{n,k} \le 4 p_n x + (k/p_n)x^{-2}$. 
Choosing $x = \lceil p_n^{-2/3} \rceil$ yields 
$p_{n,k} \le 4 (p_n^{1/3}+p_n) + kp_n^{1/3}$. The result follows. 
\end{proof}

To finish the proof of theorem~\ref{product-type}, we need to remind 
some standard facts about Vershik's first level criterion, namely 
definition 2.6, proposition 2.7 and proposition 2.17 
of~\cite{Laurent (cosiness standardness)}.   

Let $\fc = (\fc_n)_{n \le 0}$ be a filtration with independent increments 
(Laurent writes that $\fc $ is locally of product-type). 
 Given a separable metric space $(E,d)$, one says that a random variable 
$R \in L^1(\fc_0,E)$ satisfies Vershik's first level 
criterion (with respect to $\fc$) if for every $\delta>0$, 
one can find an integer $n_0 \le 0$, some innovations 
$V_{n_0+1},\ldots,V_0$ of $\fc$ at times $n_0+1,\ldots,0$ 
and some random variable $S \in L^1(\sigma(V_{n_0+1},\ldots,V_0),E)$ 
such that $\eef[d(R,S)]<\delta$. 

The subset of random variables in $L^1(\fc_0,E)$ satisfying 
Vershik's first level criterion (with respect to $\fc$) is closed
in $L^1(\fc_0,E)$. 
If $R \in L^1(\fc_0,E)$ satisfies Vershik's first level 
criterion, then any measurable real function of $R$ also satisfies 
Vershik's first level criterion. 

The first step of the proof is to check that for every $m \le 0$, 
the random variable $(X_m(1),\ldots,X_m(L))$, taking values in 
$\{a,b\}^L$ endowed with the discrete metric, 
satisfies Vershik's first level criterion with respect to 
$(\fc^{X,I}_n)_{n \le 0}$.
Indeed, by stationarity, the construction of lemma~\ref{new innovations} 
can be started at any time $n_0$ instead of time $0$. 
Starting this construction at time $n_0$ yields innovations 
$J^{n_0}_{n_0+1},J^{n_0}_{n_0+2},...$ at times $n_0+1,n_0+2,...$. 
Fix two integers $m \le 0$ and $L \ge 1$. By stationarity, 
for every $n_0 \le m$, 
$$\ppf\big[
X_m = C \circ \psi_{J^{n_0}_{n_0+1}} \circ \cdots \circ \psi_{J^{n_0}_m}
\text{ on } \odc 1,L \fdc \big] = 
\ppf\big[X_{m-n_0} = C \circ \psi_{J_1} \circ \cdots \circ \psi_{J_{m-n_0}}
\text{ on } \odc 1,L \fdc \big].$$ 
Lemma~\ref{better and better approximations} ensures that this probability
tends to $1$ as $n \to +\infty$. 

We derive successively that the following 
random variables also satisfies Vershik's first level criterion:
\begin{itemize} 
\item $X_m$, valued in $\{a,b\}^\infty$ endowed with the metric given by 
$$d(x,y) = 2^{-\inf\{i \ge 1 : x(i) \ne y(i)\}}.$$
\item $(X_m,I_{m+1},\ldots,I_0)$, valued in $\{a,b\}^\infty \times \pc(\nnf)^{|m|}$ 
endowed with the product of the metrics 
defined as above on each factor $\{a,b\}^\infty$ or $\pc(\nnf)$ 
identified with $\{0,1\}^\infty$; 
\item any measurable real function of $(X_m,I_{m+1},\ldots,I_0)$; 
\item any real random variable in $\fc^{X,I}_0$. 
\end{itemize}
The proof is complete.  
 
\subsection{A non complementable filtration yielding a 
complementable factor}

We still work with the filtration generated by the uniform randomised 
decimation process $((X_n,I_n))_{n \in \zzf}$ on the alphabet $\{a,b\}$. 
We call $\pc''(\nnf)$ the set of all infinite subsets of $\nnf$ with 
infinite complement, and we set $E = \{a,b\}^\infty \times \pc''(\nnf)$. 
Since $\nu(\pc''(\nnf))=1$, we may assume and we do assume that the 
Markov chain $((X_n,I_n))_{n \in \zzf}$ takes values in $E$. 
 
At each time $n$, we define the random variable 
$Y_n = \psi_{I_n^c}(X_{n-1})$ coding the portion of 
the infinite word $X_{n-1}$ rejected at time $n$ to get the word $X_n$. 
Of course, the knowledge of $I_n$, $X_n$ and $Y_n$ enables us to 
recover $X_{n-1}$: for every $i \in \nnf$, 
$X_{n-1}(i)$ equals $X_n(r)$ or $Y_n(r)$ according that 
$i$ is the $r^{{\rm th}}$ element of $I_n$ or of $I_n^c$. 
We can say more. 

\begin{prop}\label{rejected subsequences}
(Properties of the sequences $(Y_n)_{n \in \zzf}$ and $(I_n)_{n \in \zzf}$)
\begin{enumerate}
\item The random variables $Y_n$ are independent and uniform on 
$\{a,b\}^\infty$. 
\item The sequence $(Y_n)_{n \in \zzf}$ is independent of the sequence 
$(I_n)_{n \in \zzf}$ . 
\item Each $X_n$ is almost surely a measurable function of 
$I_{n+1},Y_{n+1},I_{n+2},Y_{n+2},...$.
\end{enumerate}
\end{prop}

Note that proposition~\ref{rejected subsequences} provides a 
constructive method to get a uniform randomised decimation 
process on $\{a,b\}$.  

\begin{proof}
The first two statements follow from a repeated application of 
lemma~\ref{complementary subwords}. 
Since the formulas involving the processes $I, X, Y$ are invariant by 
time-translations, one needs only to check the third statement when 
$n=0$. For every $i \in \nnf$, call
$$N_i = \inf\{n \ge 1 : i \notin \psi_{I_1} \circ \cdots \circ 
\psi_{I_n}(\nnf)\}.$$
the first time $n$ at which the letter 
$X_0(i)$ is rejected when forming the word $X_n$. 
For every $n \ge 0$, 
$[N_i > n] = [i \in \psi_{I_1} \circ \cdots \circ \psi_{I_n}(\nnf)]$; 
but by lemma~\ref{composition of random extractions}, 
the law of the range of $\psi_{I_1} \circ \cdots \circ \psi_{I_n}$ 
is $\nu_{2^{-n}}$, so $\ppf[N_i > n] = 2^{-n}$. 
Therefore, $N_i$ is a measurable function of $(I_n)_{n \ge 1}$ and is 
almost surely finite. On the event $[N_i<+\infty]$, 
$X_0(i) = Y_{N_i}(R_i)$, where $R_i$ is the rank of $i$ in the set 
$\psi_{I_1} \circ \cdots \circ \psi_{I_{N_i}} \circ \psi_{I^c_{N_i+1}}(\nnf)$.
The proof is complete.
\end{proof}

We split each random variable $I_n$ into two independent random 
variables, namely $U_n = \{I_n,I_n^c\}$ and $V_n = \one_{[1 \in I_n]}$. 
The random variable $U_n$ takes values in the set $\Pi$ of all 
partitions of $\nnf$ into two infinite blocks. 
Given such a partition $u \in \Pi$, we denote by $u(1)$ the block 
containing $1$ and by $u(0)$ its complement.  
Then $I_n = U_n(V_n)$ and each one of the random variables $U_n$, $U_n(0)$ 
and $U_n(1)$ carries the same information.

Call $\cc$ the cylindrical $\sigma$-field on $E^\zzf$ 
and $\pi$ the law of $((X_n,I_n))_{n \in \zzf}$. By stationarity, 
the shift operator $T$ is an automorphism of $(E^\zzf,\cc,\pi)$. 
The formulas defining $U_n,V_n,Y_n$ from $I_n$ and $X_{n-1}$ 
are invariant by time-translations so the measurable maps 
$\Phi$ and $\Psi$ yielding $(U_n)_{n \in \zzf}$ and 
$((V_n,Y_n))_{n \in \zzf}$ from $((X_n,I_n))_{n \in \zzf}$ 
commute with $T$. Therefore, the $\sigma$-fields 
$\Phi^{-1}(\cc)$ and $\Psi^{-1}(\cc)$ are factors of $T$.

\begin{theo}\label{filtration versus factor}
\begin{enumerate}
\item The factor $\Phi^{-1}(\cc)$ is complementable with complement 
$\Psi^{-1}(\cc)$. Thus, if the Markov chain $((Y_n,I_n))_{n \in \zzf}$ is 
defined on the canonical space $(E^\zzf,\cc,\pi)$, then $\fc^U_\infty$ is a 
complementable factor of $T$ with complement $\fc^{V,Y}_\infty$. 
\item Yet, the filtration $\fc^U$ is not 
complementable in the filtration $\fc^{X,I}$. 
\end{enumerate}
\end{theo}

\begin{proof} 
By proposition~\ref{rejected subsequences}, 
$\fc^U_\infty$ and $\fc^{V,Y}_\infty$ are independent, and 
$\fc^U_\infty \vee \fc^{V,Y}_\infty = \fc^{X,I}_\infty \mod \ppf$.  
Therefore, $\Phi^{-1}(\cc)$ and $\Psi^{-1}(\cc)$
are independent in $(E^\zzf,\cc,\pi)$ and 
$\Phi^{-1}(\cc) \vee \Psi^{-1}(\cc) = \cc \mod \pi$: 
the factor $\Phi^{-1}(\cc)$
is complementable with complement $\Psi^{-1}(\cc)$. 

Let $U = (U_n)_{n \le 0}$. The random variable $U$ takes values in 
$\Pi^{\zzf_-}$. 
For every $u = (u_n)_{n \le 0} \in \Pi^{\zzf_-}$ and $n \le 0$, call $W^u_n$ 
the map from $\{0,1\}^{|n|}$ to $\{a,b\}$ defined by 
$$W^u_n(v_{n+1},\ldots,v_0) 
= X_n \circ \psi_{u_{n+1}(v_{n+1})}\circ \cdots \circ \psi_{u_0(v_0)}(1).$$
By ordering the elements of $\{0,1\}^{|n|}$ in the lexicographic order, 
one identifies $W^u_n$ with an element of $\{a,b\}^{|n|}$. 

Since $X_n = X_{n-1} \circ \psi_{I_n} = X_{n-1} \circ \psi_{u_n(V_n)}$ 
$\ppf_u$-almost surely, we have  
$$W^u_n(v_{n+1},\ldots,v_0) = W^u_{n-1}(V_n,v_{n+1},\ldots,v_0)
~\ppf_u\text{-almost surely},$$
so $W^u_n$ is the left half or the right half of $W^u_{n-1}$ according 
$V_n$ equals $0$ or $1$. Moreover, under $\ppf_u$, the random variable 
$V_n$ is independent of $\fc^{W^u,V}_{n-1}$ and uniform on $\{0,1\}$.

Hence, under $\ppf_u$, the process $(W^u_n,V_n)_{n \le 0}$ is a dyadic 
split-words process with innovations $(V_n)_{n \le 0}$. The filtration 
of this process is known to be non-standard (see~\cite{Smorodinsky}). 
But one checks that $(V_n)_{n \le 0}$ is also a sequence of innovations
of the larger filtration $(\fc^{X,I})_{n \le 0}$ seen under 
$\ppf_u = \ppf[\cdot|U=u]$, 
so $(\fc^{W^u,V}_n)_{n \le 0}$ is immersed in $(\fc^{X,I})_{n \le 0}$ 
and $(\fc^{X,I})_{n \le 0}$ is also non-standard under $\ppf_u$.  
 
If $(\fc^U_n)_{n \le 0}$ admitted an independent 
complement $(\gc_n)_{n \le 0}$ in $(\fc^{X,I}_n)_{n \le 0}$, this complement 
would be immersed in the product-type filtration $(\fc^{X,I}_n)_{n \le 0}$ 
thus standard. Therefore, for $U(\ppf)$-almost every $u \in \Pi^{\zzf_-}$, 
the filtration $(\fc^{X,I}_n)_{n \le 0}$ would be standard under the probability 
$\ppf_u$, by proposition 0.1 of~\cite{Leuridan}. 
This leads to a contradiction. 
 
We are done.
\end{proof}

\section{Annex : reminders on partitions and entropy}\label{Annex}

We recall here classical definitions and results to make the paper 
self-contained. Most of them can be found in~\cite{Petersen}. 
See also~\cite{Parry}. 

In the whole section, we fix a measure-preserving map $T$ from a 
probability space $(Z,\zc,\pi)$ to itself, whereas $\alpha,\beta,\gamma$ 
denote measurable countable partitions of $Z$ 
(here, `measurable partition' means `partition into measurable blocks'), 
and $\fc,\gc$ denote sub-$\sigma$-fields of $\zc$. 

We will use the non-negative, continuous and strictly 
concave function $\varphi : [0,1] \to \rrf$ defined by 
$\varphi(x) = -x\log_2(x)$, with the convention $\varphi(0)=0$. 
The maximum of this function is $\varphi(1/e) = 1/(e\ln2)$.

\subsection{Partitions}

Defining the entropy requires discretizations of the ambient 
probability space, that is why we introduce countable measurable 
partitions. Equivalently, we could use discrete random variables. 
We need a few basic definitions.   

\begin{defi}
One says that $\beta$ is finer than $\alpha$ 
(and note $\alpha \le \beta$) when each block 
of $\alpha$ is the union of some collection of blocks of $\beta$, 
i.e. when $\sigma(\alpha) \subset \sigma(\beta)$. 
\end{defi}

\begin{defi}
The (non-empty) intersections $A \cap B$ with $A \in \alpha$ 
and $B \in \beta$ form a partition; this partition is the coarsest 
refinement of $\alpha$ and $\beta$ and is denoted by $\alpha \vee \beta$. 
\end{defi}

\begin{defi}
More generally, if $(\alpha_k)_{k \in K}$ is a countable family of countable 
measurable partitions of $Z$, we denote by $\bigvee_{k \in K}\alpha_k$ the 
partition whose blocks are the (non-empty) intersections 
$\bigcap_{k \in K}A_k$ where $A_k \in \alpha_k$ for every $k \in K$; this 
partition is the coarsest refinement of the $(\alpha_k)_{k \in K}$; it is 
still measurable but it can be uncountable. 
\end{defi}

\begin{defi}
The partitions $\alpha$ and $\beta$ are independent if and only if
$\pi(A \cap B) = \pi(A) \pi(B)$ for every $A \in \alpha$ 
and $B \in \beta$.
\end{defi}

\begin{defi}
We denote by $T^{-1}\alpha$ the partition defined by 
$$T^{-1}\alpha = \{T^{-1}(A) : A \in \alpha\}.$$
If $T$ is invertible (i.e. bimeasurable), we denote 
by $T\alpha$ the partition defined by 
$$T\alpha = \{T(A) : A \in \alpha\}.$$
\end{defi}

\subsection{Fischer information and entropy of a partition}

Given $A \in \zc$, we view $-\log_2 \pi(A)$ as the quantity of information 
provided by the event $A$ when $A$ occurs, with the convention 
$-\log_20 = +\infty$.~\footnote{Taking logarithms in base $2$ is an 
arbitrary convention which associates one unity of information to any 
uniform Bernoulli random variable.}
With this definition, the occurence of a rare event 
provide much information; moreover, the information provided 
by two independent events $A$ and $B$ occuring at the same time is 
the sum of the informations provided by each of them separately. 
The entropy of a countable measurable partition is the mean quantity 
of information provided by the blocks in it.

\begin{defi}
The Fischer information of the partition $\alpha$ is the random variable  
$$I_\alpha := \sum_{A \in \alpha} (-\log_2\pi(A)) \one_A.$$ 
The entropy of the partition $\alpha$ is the quantity
$$H(\alpha) = \eef_\pi[I_\alpha] = \sum_{A \in \alpha} \varphi(\pi(A)).$$ 
\end{defi} 

Note that null blocks in $\alpha$ do not give any contribution to the entropy 
of a partition. Non-trivial partitions have positive entropy. Finite 
partitions have finite entropy. Infinite countable partition can have 
finite or infinite entropy. 

The previous definition can be generalized as follows.

\begin{defi}
The conditional Fischer information of the partition $\alpha$ 
with regard to $\fc$ is the random variable 
$$I_{\alpha|\fc} = \sum_{A \in \alpha} (-\log_2\pi(A|\fc))\one_A .$$ 
The conditional entropy of the partition $\alpha$ with regard to $\fc$ 
is the quantity
$$H(\alpha|\fc) = \eef_\pi[I_{\alpha|\fc}].$$ 
\end{defi}
 
\begin{rema}\label{alternative formula}
By conditional Beppo-Levi theorem,
$$\eef[I_{\alpha|\fc}|\fc] = \sum_{A \in \alpha} \varphi(\pi(A|\fc)),$$
so
$$H(\alpha|\fc) = \sum_{A \in \alpha} \eef_\pi[\varphi(\pi(A|\fc))].$$ 
\end{rema} 

Given any partition $\eta$ into measurable blocks, we will use the 
following abbreviated notations:
$H(\alpha|\eta):=H(\alpha|\sigma(\eta))$,  
$H(\alpha|\eta \vee \fc):=H(\alpha|\sigma(\eta) \vee \fc)$.

Note that when $\fc$ is the trivial $\sigma$-field $\{\emptyset,Z\}$, 
$I_{\alpha|\fc}$ and $H(\alpha|\fc)$ are equal to $I_{\alpha}$ and $H(\alpha)$.

The following properties are very useful and are checked by direct 
computation, by using the positivity of Fischer information and the 
strict concavity of $\varphi$.
  
\begin{prop}\label{properties of conditional entropy}(First properties)
\begin{enumerate}
\item $I_{T^{-1}\alpha|T^{-1}\fc} = I_{\alpha|\fc} \circ T$ so
$H(T^{-1}\alpha|T^{-1}\fc)=H(\alpha|\fc)$. 
\item $H(\alpha|\fc) \ge 0$, with equality if and only if 
$\alpha \subset \fc \mod \pi$. 
\item $H(\alpha|\fc) \le H(\alpha)$. When $H(\alpha)<+\infty$, 
equality holds if and only if $\alpha$ is independent of $\fc$. 
\item If $\fc \subset \gc$, then 
$\eef[I_{\alpha|\fc}|\fc] \ge \eef[I_{\alpha|\gc}|\fc]$ so
$H(\alpha|\fc) \ge H(\alpha|\gc)$. 
\item If $\alpha \le \beta$, then $I_{\alpha|\fc} \le I_{\beta|\fc}$ so
$H(\alpha|\fc) \le H(\beta|\fc)$.
\item $H(\alpha \vee \beta|\fc) 
= H(\alpha|\fc) + H(\beta|\fc \vee \alpha) 
\le H(\alpha|\fc) + H(\beta|\fc)$.
\end{enumerate}
\end{prop}

The last item above (addition formula above and sub-additivity of entropy) 
is used repeatedly in the present paper. We will also use the next result.  

\begin{prop}
\label{conditional entropy and monotone sequence of sigma-fields}
(monotone sequence of $\sigma$-fields). Assume that $H(\alpha)<+\infty$.
\begin{enumerate}
\item If $(\fc_n)_{n \ge 0}$ is a non-decreasing sequence of $\sigma$-fields, 
then 
$$H(\alpha | \fc_n) \to H(\alpha | \fc_\infty) 
\text{ where } \fc_\infty = \bigvee_{n \ge 0} \fc_n.$$
\item If $(\dc_n)_{n \ge 0}$ is a non-increasing sequence of $\sigma$-fields, 
then 
$$H(\alpha | \dc_n) \to H(\alpha | \dc_\infty) 
\text{ where } \dc_\infty = \bigcap_{n \ge 0} \dc_n.$$
\end{enumerate}
\end{prop}

\begin{proof}


Given $A \in \alpha$, the martingale and backward martingale 
convergence theorems and the continuity of $\varphi$ yield 
$\varphi(\pi(A|\fc_n)) \to \varphi(\pi(A|\fc_\infty))$
and $\varphi(\pi(A|\dc_n)) \to \varphi(\pi(A|\dc_\infty))$ 
as $n \to +\infty$. 
When the partition $\alpha$ is finite, the convergences 
$H(\alpha | \fc_n) \to H(\alpha | \fc_\infty)$ and 
$H(\alpha | \dc_n) \to H(\alpha | \dc_\infty)$ follow 
by remark~\ref{alternative formula}. 

The result can be extended to the general case by approximating $\alpha$ 
with finite mesurable partitions and using the equicontinuity of the maps 
$\alpha \mapsto \eef[\alpha|\fc]$, where $\fc$ is any sub-$\sigma$-field 
of $\zc$. See propositions~\ref{density} and~\ref{continuity of entropy}
in the next subsection. 
\end{proof}

\subsection{Continuity properties}

\begin{prop}
The formula 
$$d(\alpha,\beta) = H(\alpha|\beta) + H(\beta|\alpha) 
= 2H(\alpha\vee\beta) - H(\alpha) - H(\beta)$$
defines a pseudo-metric on the set of all partitions of $Z$
with finite entropy. 
Moreover, $d(\alpha,\beta) = 0$ if and only if $\sigma(\alpha) 
= \sigma(\beta)$ modulo $\pi$ (i.e. the non-null blocks of 
$\alpha$ and $\beta$ are the same modulo $\pi$).
\end{prop}

\begin{proof}
The triangle inequality follows from the inequality
$$H(\alpha|\gamma) 
\le H(\alpha\vee\beta|\gamma) 
= H(\beta|\gamma) + H(\alpha|\beta\vee\gamma)
\le H(\beta|\gamma) + H(\alpha|\beta).$$
The other statements follow from 
proposition~\ref{properties of conditional entropy}.
\end{proof}

\begin{prop}~\label{density}
For the pseudo-metric $d$ thus defined, 
the set of all finite measurable partitions of $Z$ is dense in 
the set of all (measurable) partitions on $Z$ with finite entropy.
\end{prop}

\begin{proof}
Let $\alpha = \{A_n : n \ge 1\}$ be an infinite partition of
$Z$ with finite entropy. For every $n \ge 1$, set 
$\alpha_n = \{A_1,\cdots,A_n,(A_1 \cup \cdots \cup A_n)^c\}$. 
Since $\alpha$ is finer than $\alpha_n$, 
$$H(\alpha) \ge H(\alpha_n) \ge \sum_{k=1}^n \varphi(\pi(A_k)),$$
so $d(\alpha,\alpha_n) = H(\alpha)-H(\alpha_n) \to 0$ as $n \to +\infty$. 
\end{proof}

\begin{prop}~\label{continuity of entropy}
Let $\fc$ be a sub-$\sigma$-field of $\zc$. Then, for the pseudo-metric $d$, 
the map $\alpha \mapsto H(\alpha|\fc)$ is $1$-Lipschitz.  
\end{prop}
 
\begin{proof}
Let $\alpha$ and $\beta$ be two partitions of $Z$ with finite entropy.
Then
\begin{eqnarray*}
H(\beta|\fc) - H(\alpha|\fc) 
\le H(\alpha \vee \beta|\fc) - H(\alpha|\fc) 
= H(\beta | \fc \vee \alpha) 
\le H(\beta | \alpha) 
\le d(\alpha,\beta).
\end{eqnarray*}
The result follows.
\end{proof}

\begin{prop}\label{continuity w.r.t. the blocks} 
Let $\alpha = \{A_1,\ldots,A_n\}$ and $\beta = \{B_1,\ldots,B_n\}$ 
be two finite measurable partitions of $Z$ with the same finite number 
of blocks. 
For every $A$ and $B$ in $\zc$, set $\delta(A,B) = \pi(A \triangle B)$.Then 
$$d(\alpha,\beta) 
\le \sum_{i=1}^n 2\varphi\big(\delta(A_i,B_i)/2)\big)
+ \sum_{i=1}^n\delta(A_i,B_i)/\ln2.$$
Therefore, the partition $\alpha$ depends continuously on the blocks 
$A_1,\ldots,A_n$. In particular, the map $A \mapsto \{A,A^c\}$ is 
(uniformly) continuous for the pseudo-metrics $\delta$ and $d$. 
\end{prop}


\begin{proof}
Fix $i \in \odc 1,n \fdc$. Then the concavity of $\varphi$ yields 
\begin{eqnarray*}
\sum_{j \ne i} \pi(B_j) \varphi(\pi(A_i|B_j))
&\le& \pi(B_i^c) \varphi 
\Big( \sum_{j \ne i} \frac{\pi(B_j)}{\pi(B_i^c)} \pi(A_i|B_j) \Big) \\
&=& \pi(B_i^c) \varphi 
\Big( \sum_{j \ne i} \frac{\pi(A_i \cap B_j)}{\pi(B_i^c)} \Big) \\
&=& \pi(B_i^c) \varphi 
\Big( \frac{\pi(A_i \cap B_i^c)}{\pi(B_i^c)} \Big) \\
&=& \pi(A_i \cap B_i^c) [-\log_2\pi(A_i \cap B_i^c)+\log_2\pi(B_i^c)] \\
&\le& \varphi(\pi(A_i \cap B_i^c)).
\end{eqnarray*}
But the concavity of $\varphi$ also yields $\varphi(x) \le (1-x)/\ln2$ 
for every $x \in [0,1]$, so
$$\pi(B_i) \varphi(\pi(A_i|B_i)) 
\le \pi(B_i) \pi(A_i^c|B_i) / \ln 2 
= \pi(A_i^c \cap B_i) / \ln 2.$$
Hence, by remark~\ref{alternative formula}, 
$$H(\alpha|\beta) = \sum_{i,j} \pi(B_j) \varphi(\pi(A_i|B_j))
\le \sum_i \varphi(\pi(A_i \cap B_i^c)) + \sum_i \pi(A_i^c \cap B_i) / \ln 2.$$
A similar upper bound holds for $H(\beta|\alpha)$. Summing these two 
inequalities and using once again the concavity of $\varphi$ yields the 
statement.   
\end{proof}

\subsection{Entropy of a measure-preserving map}

First, we define quantities $h(T,\alpha)$.

\begin{prop}\label{entropy along a partition}
\begin{enumerate}(Definition and formula for $h(T,\alpha)$)
\item The sequence $(H_n(T,\alpha))_{n \ge 0}$ defined by 
$$H_n(T,\alpha) 
= H(\alpha \vee T^{-1}\alpha \vee \cdots \vee T^{-(n-1)}\alpha)$$
is concave. Since $H_0(T,\alpha)=0$, the sequence 
$(H_n(T,\alpha)/n)_{n \ge 1}$ is non-increasing so the 
limit $h(T,\alpha) = \lim_{n \to +\infty} H_n(T,\alpha)/n$ 
exists in $[0,+\infty]$. 
\item If $H(\alpha)<+\infty$, then $h(T,\alpha) = H(\alpha | \alpha_1^\infty)$, 
where $\alpha_1^\infty = \bigvee_{k \ge 1} T^{-k}\alpha$. 
\end{enumerate}
\end{prop}

\begin{proof}
The first statement follows from the equality 
\begin{eqnarray*}
H_{n+1}(T,\alpha) - H_n(T,\alpha) 
&=& H(\alpha \vee T^{-1}\alpha \vee \cdots \vee T^{-n}\alpha)
-H (T^{-1}\alpha \vee \cdots \vee T^{-n}\alpha) \\
&=& H(\alpha | T^{-1}\alpha \vee \cdots \vee T^{-(n-1)}\alpha), 
\end{eqnarray*}
and the fact that $H(\alpha|\fc)$ is non-increasing with regard to $\fc$. 
Using 
proposition~\ref{conditional entropy and monotone sequence of sigma-fields} 
and Ces\`aro's lemma yields the second statement.
\end{proof}

\begin{defi}
The entropy of $T$ is 
\begin{eqnarray*}
h(T) 
&=& \sup\{h(T,\alpha) : \alpha 
\mathrm{~partition~of~} Z \mathrm{~with~finite~entropy}\} \\
&=& \sup\{h(T,\alpha) : \alpha 
\mathrm{~finite~measurable~partition~of}~Z\}.
\end{eqnarray*}
These two quantities coincide thanks to propositions~\ref{density} 
and~\ref{continuity}.
\end{defi}

\begin{prop}
For every $r \ge 1$, $h(T^r) = rh(T)$. 
If $T$ is also invertible, one has also $h(T^{-1}) = h(T)$. 
\end{prop}

\begin{proof}
For every $n \ge 1$ and every partition $\alpha$ with finite entropy, 
$$H_n(T^r,\alpha) \le H_n(T^r,\alpha \vee \cdots \vee T^{-(r-1)}\alpha) 
= H_{rn}(T,\alpha).$$
Dividing by $n$ and letting $n$ go to infinity yields 
$$h(T^r,\alpha) \le h(T^r,\alpha \vee \cdots \vee T^{-(r-1)}\alpha) 
= rh(T,\alpha).$$
The inequalities $h(T^r) \le rh(T)$ and $rh(T) \le h(T^r)$ follow.

If $T$ is invertible, the equalities 
$\alpha \vee \cdots \vee T^{-(n-1)}\alpha 
= T^{-(n-1)}(\alpha \vee \cdots \vee T^{{n-1}}\alpha)$
follow from proposition~\ref{properties of conditional entropy} item~1 and
yield $H_n(T,\alpha) = H_n(T^{-1},\alpha)$, so $h(T^{-1}) = h(T)$. 
\end{proof}

\subsection{Generators}

Countable generating partitions help us to compute the entropy 
of invertible measure-preserving maps. 

\begin{defi}
Assume that $T$ is invertible. A countable measurable partition 
$\gamma$ is generating (with regard to $T$) if the partitions 
$(T^k\gamma)_{k \in \zzf}$ generate $\zc$ modulo the null sets. 
\end{defi}

\begin{theo}\label{Kolmogorov - Sinai}(Kolmogorov - Sinai theorem)
If $T$ is invertible and $\gamma$ is a countable generator 
(with regard to $T$), then $h(T) = h(T,\gamma)$. 
\end{theo}

In the next subsection, we will prove a conditional version of this 
classical theorem, namely theorem~\ref{conditional Kolmogorov - Sinai}.  

Here is the basic example of generator. 

\begin{exam}
Let $\Lambda$ be a countable set, $p_0 : (y_k)_{k \in \zzf} \mapsto y_0$ 
the $0$-coordinate projection from $\Lambda^\zzf$ to $\Lambda$, 
$S : (y_k)_{k \in \zzf} \mapsto (y_{k+1})_{k \in \zzf}$ the shift operator 
on $\Lambda^\zzf$, and $\mu$ any shift-invariant probability measure 
on $\Lambda^\zzf$. Then the partition 
$\big\{p_0^{-1}\{\lambda\} : \lambda \in \Lambda\big\}$ 
is generating with regard to $S$.
\end{exam}

The interesting fact is that many situations can be reduced to this 
particular case. The proof of next theorem is outlined 
in~\cite{Kalikow - McCutcheon}.

\begin{theo}(Rohlin's countable generator theorem)
If $(Z,\zc,\pi)$ is a Lebesgue space, $T$ is invertible and 
{\rm aperiodic}, i.e. if $\pi\{z \in Z : \exists n \ge 1 : T^n(z)=z\}=0$,
then $T$ admits a countable generating partition 
$\gamma = \{C_\lambda : \lambda \in \Lambda\}$. Moreover, 
the $\gamma$-name map $\Phi$ from $Z$ to $\Lambda^\zzf$, defined by 
$\Phi(z)_k = \lambda$ whenever $T^k(z) \in C_\lambda$,
is invertible modulo the null sets, when $\Lambda^\zzf$ 
is endowed with the probability measure $\Phi(\pi)$. 
The measure $\Phi(\pi)$ is shift-invariant, so $T$ is isomorphic 
modulo the null sets to the shift operator on $\Lambda^\zzf$.
\end{theo}

When $T$ is invertible, ergodic and has finite entropy, 
Krieger's theorem ensures the existence of a finite 
generator with size at most $\lfloor 2^{h(T)} \rfloor +1$. 
We do not use this refinement in the present paper. 

Using the remark given in footnote in subsection~\ref{factors}, 
one checks that if $T$ is invertible and $(Z,\zc,\pi)$ is a Lebesgue space, 
then any factor of $T$ admits a countable generating partition.

\subsection{Conditional entropy given a factor. Pinsker's formula}

Assume that $T$ is invertible and that $\bc$ is a factor of $T$.
One may define the entropy of $T$ given $\bc$ as follows.   

\begin{prop}\label{conditional entropy given a factor}
\begin{enumerate}(Definition and formula for $h(T,\alpha|\bc)$)
\item The sequence $(H_n(T,\alpha|\bc))_{n \ge 0}$ defined by
$$H_n(T,\alpha|\bc) = 
H(\alpha \vee T^{-1}\alpha \vee \cdots \vee T^{-(n-1)}\alpha|\bc)$$
is concave. Since $H_0(T,\alpha|\bc) = 0$, the sequence 
$(H_n(T,\alpha|\bc)/n)_{n \ge 1}$ is non-increa\-sing so the limit 
$h(T,\alpha|\bc) = \lim_{n \to +\infty} H_n(T,\alpha|\bc)/n$ 
exists in $[0,+\infty]$. 
\item If $H(\alpha)<+\infty$, then $h(T,\alpha|\bc) 
= H(\alpha | \alpha_1^\infty\vee \bc)$, 
where $\alpha_1^\infty = \bigvee_{k \ge 1} T^{-k}\alpha$ denotes the $\sigma$-field 
generated by the partitions $(T^{-k}\alpha)_{k \ge 1}$. 
\end{enumerate}
\end{prop}

\begin{proof}
Since $T^{-1}\bc=\bc$, one has $H_n(T,\alpha|\bc) 
= H (T^{-1}\alpha \vee \cdots \vee T^{-n}\alpha|\bc)$, so
\begin{eqnarray*}
H_{n+1}(T,\alpha|\bc) - H_n(T,\alpha|\bc) 
&=& H(\alpha \vee T^{-1}\alpha \vee \cdots \vee T^{-n}\alpha|\bc) \\
& & \quad -H (T^{-1}\alpha \vee \cdots \vee T^{-n}\alpha|\bc) \\
&=& H(\alpha | \sigma(T^{-1}\alpha \vee \cdots \vee T^{-(n-1)}\alpha \vee \bc). 
\end{eqnarray*}
The statements follow, by  
proposition~\ref{conditional entropy and monotone sequence of sigma-fields} 
and Ces\`aro's lemma.
\end{proof}

\begin{prop}~\label{continuity}
Assume that $T$ is invertible and that $\bc$ is a factor of $T$. 
If $\alpha$ and $\gamma$ are two partitions of $Z$ with finite entropy, then 
$h(T,\alpha|\bc) - h(T,\gamma|\bc) \le H(\alpha | \gamma) \le d(\alpha,\gamma)$.
Therefore, for the pseudo-metric $d$, the map 
$\alpha \mapsto h(T,\alpha|\bc)$ is $1$-Lipschitz.  
\end{prop}
 
\begin{proof}
Set 
$\alpha_0^{n-1} = \alpha \vee T^{-1}\alpha \vee \cdots \vee T^{-(n-1)}\alpha$ and
$\gamma_0^{n-1} = \gamma \vee T^{-1}\gamma \vee \cdots \vee T^{-(n-1)}\gamma$
for every $n \ge 1$. Then
\begin{eqnarray*}
H(\alpha_0^{n-1} |\bc) - H(\gamma_0^{n-1}|\bc) 
&\le& H(\alpha_0^{n-1} \vee \gamma_0^{n-1}|\bc) - H(\gamma_0^{n-1}|\bc) \\
&=& H(\alpha_0^{n-1} | \bc \vee \gamma_0^{n-1}) \\
&\le& \sum_{k=0}^{n-1} H(T^{-k}\alpha | \bc \vee \gamma_0^{n-1}) \\
&\le& \sum_{k=0}^{n-1} H(T^{-k}\alpha | T^{-k}\gamma) \\
&=& nH(\alpha | \gamma).
\end{eqnarray*}
Dividing by $n$ and letting $n$ go to infinity yields
$h(T,\alpha|\bc) - h(T,\gamma|\bc) \le H(\alpha | \gamma) \le d(\alpha,\gamma)$.
The result follows.
\end{proof}

\begin{defi}
The conditional entropy of $T$ given $\bc$ is the quantity 
\begin{eqnarray*}
h(T|\bc) 
&=& \sup\{h(T,\alpha|\bc) : \alpha 
\mathrm{~partition~of~} Z \mathrm{~with~finite~entropy}\} \\
&=& \sup\{h(T,\alpha|\bc) : \alpha 
\mathrm{~finite~measurable~partition~of}~Z\}.
\end{eqnarray*}
These two quantities coincide thanks to propositions~\ref{density} 
and~\ref{continuity}.
\end{defi}

Kolmogorov - Sinai theorem admits the following generalization.

\begin{theo}\label{conditional Kolmogorov - Sinai}
If $\gamma$ is a countable generator of $T$, then 
$h(T|\bc) = h(T,\gamma|\bc)$.
\end{theo}

\begin{proof}
For every integers $p \le q$, set 
$$\gamma_p^q = \bigvee_{k=p}^q T^{-k}\gamma.$$
Fix $r \ge 0$. Then for every integer $n \ge 1$, 
$$\frac{1}{n} H_n(T,\gamma_{-r}^r|\bc)
= \frac{1}{n} H(T,\gamma_{-r-n+1}^r|\bc)
= \frac{n+2r}{n} \times \frac{1}{n+2r} H(T,\gamma_{-r-n+1}^r|\bc).$$
Letting $n$ go to infinity yields $h(T,\gamma_{-r}^r|\bc) = h(T,\gamma|\bc)$.
Thus, applying proposition~\ref{continuity} any partition $\alpha$ of $Z$ 
with finite entropy and yo $\gamma_{-r}^r$ yields 
$$h(T,\alpha|\bc) - h(T,\gamma|\bc) \le H(\alpha | \gamma_{-r}^r).$$
But $H(\alpha | \gamma_{-r}^r) \to H(\alpha | \zc) = 0$ as $r \to +\infty$ 
since $\gamma$ is a countable generator of $T$. Hence 
$h(T,\alpha|\bc) \le h(T,\gamma|\bc)$. The conclusion follows. 
\end{proof}

\begin{prop}\label{Pinsker's formula}(Pinsker's formula)
Assume that $\alpha$ and $\beta$ have finite entropy. 
Let $\ac$ and $\bc$ be two factors generated by $\alpha$ and $\beta$. 
Set $\alpha_1^\infty = \bigvee_{k \ge 1} T^{-k}\alpha$ and 
$\beta_1^\infty = \bigvee_{k \ge 1} T^{-k}\beta$. Then  
$$h((T,\ac)|\bc) = h(T,\ac \vee \bc) - h(T,\bc),$$
or equivalently, 
$$H(\alpha | \alpha_1^\infty\vee \bc) 
= H(\alpha \vee \beta | \alpha_1^\infty\vee \beta_1^\infty) - H(\beta | \beta_1^\infty).$$ 
\end{prop}

\begin{proof}
For every integers $p \le q$, define the partitions $\alpha_p^q$ and 
$\beta_p^q$ like in the proof above.  
Then for every non-negative integer $n$, 
\begin{eqnarray*}
H_{n+1}(T,\alpha \vee \beta) - H_{n+1}(T,\beta) 
&=& H(\alpha_{-n}^0 \vee \beta_{-n}^0) - H(\beta_{-n}^0) \\
&=& \sum_{k=0}^{n} H(T^k\alpha | \alpha_{-(k-1)}^0 \vee \beta_{-n}^0) \\
&=& \sum_{k=0}^n H(\alpha | \alpha_1^k \vee \beta_{k-n}^k).
\end{eqnarray*}
By proposition~\ref{conditional entropy and monotone sequence of sigma-fields}, 
$H(\alpha | \alpha_1^k \vee \beta_{-\ell}^k) \to 
H(\alpha | \alpha_1^\infty \vee \bc)$ as $k \to +\infty$ and $\ell \to +\infty$. 
Since the quantities $H(\alpha | \alpha_1^k \vee \beta_{k-n}^k)$ 
belong to the finite interval $[0,H(\alpha)]$, one deduces that  
$$h(T,\alpha \vee \beta) - h(T,\beta) 
= \lim_{n \to +\infty} \frac{1}{n+1} 
\big( H_{n+1}(T,\alpha \vee \beta) - H_{n+1}(T,\beta) \big)
= H(\alpha | \alpha_1^\infty\vee \bc).$$ 
Hence, the statement follows from proposition~\ref{entropy along a partition}, 
theorem~\ref{Kolmogorov - Sinai},
proposition~\ref{conditional entropy given a factor}
and theorem~\ref{conditional Kolmogorov - Sinai}.
\end{proof}

\section*{Acknowledgements}

I thank A.~Coquio, J.~Brossard, M.~\'Emery, S.~Laurent, J.P.~Thouvenot for their useful remarks and for stimulating conversations.

\end{document}